\documentclass{article}
\usepackage{amsfonts}
\usepackage{amsmath}
\usepackage{amssymb}
\usepackage{amsthm}
\usepackage{enumerate}
\usepackage[latin1]{inputenc} 
\usepackage{color}
\usepackage{mathtools}
\usepackage{mathdots}
\usepackage{amsmath,amsthm,amssymb}
\usepackage{cite}
\usepackage{authblk}
\usepackage{booktabs}
\usepackage{mathabx}
\usepackage{multirow}
\usepackage{multicol}
\usepackage{tikz}
\usepackage{subdepth}
\usepackage{mathtools}
\usepackage{bbold}

\usepackage{hyperref}

\newtheorem{theorem}{Theorem}

\newtheorem{lemma}[theorem]{Lemma}
\theoremstyle{definition}
\newtheorem{example}[theorem]{Example}

\newtheorem{remark}[theorem]{Remark}
\newtheorem{definition}[theorem]{Definition}

\newcommand{\CC}{{\mathbb C}}

\newcommand{\rank}{{\mbox{\rm rank\,}}}
\newcommand{\diag}{{\mbox{\rm diag}}}

\newcommand{\la}{\lambda}
\newcommand{\orb}{{\cal O}}
\newcommand{\bun}{{\cal B}}
\newcommand{\gln}{{\rm GL}_n(\CC)}
\newcommand{\glm}{{\rm GL}_m(\CC)}

\title{On bundles of matrix pencils under strict equivalence\footnote{ This work has been supported by the Agencia Estatal de Investigaci\'on of Spain through grants PID2019-106362GB-I00 MCIN/
AEI/10.13039/501100011033/ and MTM2017-90682-REDT, and by the Madrid Government (Comunidad de Madrid-Spain) under the Multiannual Agreement with UC3M in the line of Excellence of University Professors (EPUC3M23), and in the context of the V PRICIT (Regional Programme of Research and Technological Innovation).}}
\author{Fernando De Ter\'an\thanks{Departamento de Matem\'aticas, Universidad Carlos III de Madrid, Avda. Universidad 30, 28911 Legan\'es, Spain. \text{fteran@math.uc3m.es}},  Froil\'an M. Dopico\thanks{Departamento de Matem\'aticas, Universidad Carlos III de Madrid, Avda. Universidad 30, 28911 Legan\'es, Spain. \text{dopico@math.uc3m.es}}}
\date{\today}

\begin{document}

\maketitle

\begin{abstract}
   Bundles of matrix pencils (under strict equivalence) are sets of pencils having the same Kronecker canonical form, up to the eigenvalues (namely, they are an infinite union of orbits under strict equivalence). The notion of bundle for matrix pencils was introduced in the 1990's, following the same notion for matrices under similarity, introduced by Arnold in 1971, and it has been extensively used since then. Despite the amount of literature devoted to describing the topology of bundles of matrix pencils, some relevant questions remain still open in this context. For example, the following two: (a) provide a characterization for the inclusion relation between the closures (in the standard topology) of bundles; and (b) are the bundles open in their closure? The main goal of this paper is providing an explicit answer to these two questions. In order to get this answer, we also review and/or formalize some notions and results already existing in the literature. We also prove that bundles of matrices under similarity, as well as bundles of matrix polynomials (defined as the set of $m\times n$ matrix polynomials of the same grade having the same spectral information, up to the eigenvalues) are open in their closure.
\end{abstract}

\noindent{\bf Keywords}. Matrix, matrix pencil, matrix polynomial, spectral information, strict equivalence, Kronecker canonical form, Jordan canonical form, orbit, bundle, open set, closure, majorization.

\noindent{\bf AMS Subject Classification}. 15A22, 15A18, 15A21, 15A54, 65F15.

\section{Introduction}\label{intro_sec}

Orbits of matrices and matrix pencils arise as a natural object when dealing with equivalence relations and their canonical forms. Orbits of matrices under similarity were introduced by Arnold in \cite{arnold1971}, and in this case the canonical form is the well-known {\em Jordan Canonical Form} (JCF). More precisely, the orbit (under similarity) of a given matrix consists of all matrices with the same JCF. In other words, it is the orbit of the given matrix (with size, say, $n\times n$) under the action of similarity of the general linear group $\gln$, consisting of all $n\times n$ invertible matrices with complex entries, on the set of all $n\times n$ matrices with complex entries, $\CC^{n\times n}$,  namely 
$$
\begin{array}{ccc}
     \gln\times \CC^{n\times n}&\rightarrow&\CC^{n\times n}  \\
     (P,A)&\mapsto&PAP^{-1}.
\end{array}
$$
In the case of general (unstructured) matrix pencils (namely, pairs of matrices of size $m\times n$), similarity is replaced by the so-called {\em strict equivalence}, that is
$$
\begin{array}{ccc}
     \glm\times\gln\times \CC^{m\times n}\times\CC^{m\times n} &\rightarrow&\CC^{m\times n}\times \CC^{m\times n}  \\
     (P,Q,A,B)&\mapsto&(PAQ,PBQ),
\end{array}
$$
and the canonical form under this relation is the {\em Kronecker Canonical Form} (KCF) (see \cite[Ch. XII, \S5]{gant}). Thus the orbit of a given pencil consists of all pencils with the same KCF.

The notion of orbit has some limitations when studying the change of the spectral information, because the eigenvalues of the matrix or the matrix pencil must be fixed. But if, for instance, we are interested in analyzing the change of the spectral information under small changes in the entries of the matrices $A$ and $B$ above, we should allow the eigenvalues to change, since this is what happens under generic perturbations. In order to overcome this limitation, Arnold introduced the notion of {\em bundle} in his 1971 paper \cite{arnold1971}. A bundle is the union (usually infinite) of all orbits that have the same JCF (or the same KCF for pencils) up to the specific values of the eigenvalues. More precisely, and taking the words by Arnold for the case of matrices, ``a bundle is the set of all matrices whose
Jordan normal forms differ only by their eigenvalues, but for which the sets
of distinct eigenvalues and the orders of the Jordan blocks are the same". For instance, for the matrix
$$
A=\left[\begin{array}{cc|c}\lambda_0&1&0\\0&\lambda_0&0\\\hline 0&0&\mu_0\end{array}
\right],
$$
consisting of one Jordan block with size $2$ associated with the eigenvalue $\lambda_0$ and one Jordan block with size $1$ associated with the eigenvalue $\mu_0$, with $\lambda_0\neq\mu_0$, the corresponding bundle is the set
$$
{\cal B}(A):=\left\{P\left[\begin{array}{cc|c}\lambda&1&0\\0&\lambda&0\\\hline 0&0&\mu\end{array}
\right]P^{-1}:\ \lambda,\mu\in\CC,\ \lambda\neq\mu,\ P\in\mbox{GL}_3(\CC)\right\}.
$$
The notion of bundle can be easily extended to matrix pencils, namely a bundle is the set of matrix pencils with the same KCF, up to the eigenvalues. Bundles of matrix pencils have been considered in many papers, mostly in the last $30$ years, like \cite{ddd-pencils,dk14,eek1,eek2,ejk,ejk2003,pervouchine,starcic}. Some of these references deal with bundles of {\em structured} pencils, namely those enjoying some particular symmetry in the coefficient matrices, $A$ and $B$, of the pencil $(A,B)$ (like {\em alternating}, {\em (skew-) Hermitian}, {\em (anti-)palindromic}, or {\em (skew-)symmetric}, see, for instance, \cite{4m-good} for the definition of all these structures), and in these cases the strict equivalence relation is replaced by the congruence or $*$-congruence relation. Even bundles of unstructured matrices under the relation of congruence (or $*$-congruence) have been considered in the literature, mostly arising from structured matrix pencils like, for instance, in \cite{dd11,dfs14,dfkks}. However, structured pencils are out of the analysis carried out in the present work, and we consider only general (unstructured) matrix pencils.

The canonical forms mentioned above (namely, the JCF for matrices or the KCF for matrix pencils) are the representatives of the equivalence classes of matrices or matrix pencils under the action of similarity or strict equivalence (or congruence/$*$-congruence in the structured case), respectively, (see, for instance, \cite{eek2} for more information on these canonical forms in the unstructured case, and \cite{thompson} for structured matrix pencils). A relevant question, not only theoretically, but also for applied and numerical purposes (for instance, in the computation of the canonical forms) is to determine which are the most ``likely" representatives of these canonical forms. Or, more precisely, given two different representatives of the canonical form, to determine whether one of them is more ``likely" than the other or not. In order to answer this question, the standard approach, that has been followed, for instance, in \cite{ddd-pencils,ddd-polys,d17,dd17,dd-skew,dfkks,dk14,eek1,eek2,ejk,pervouchine} is to translate it to the context of orbits or, more in general, of bundles. To be more precise, and focusing on strict equivalence of matrix pencils, each bundle is associated with a particular pencil $L(\la)$ or, more in general, with the KCF of $L(\la)$, which is determined up to the values of the eigenvalues. Then, given two bundles $\bun(L_1)$ and $\bun(L_2)$, we say that the KCF of $L_1(\la)$ is ``more likely" (or ``more generic") than the one of $L_2(\la)$ if the closure of the bundle of $L_2(\la)$ is included in the closure of the bundle of $L_1(\la)$ (namely $\overline\bun(L_2)\subseteq\overline\bun(L_1)$, in the notation that is used throughout the manuscript), where the closures are considered in the standard topology of $\CC^{m\times n}\times\CC^{m\times n}$, which is identified with $\CC^{2mn}$. Therefore, the inclusion relations of bundle closures determine the ``likelihood" of the corresponding canonical forms. 

For the relation of similarity of matrices (namely, the likelihood of the JCF), a characterization of the inclusion relation between orbit closures is known since the 1980's \cite{bt80,mp83}. A characterization for the inclusion relation between orbit closures of matrix pencils under strict equivalence was obtained in \cite{bongartz} and \cite{pokrzywa}, and for bundle closures of matrices under similarity, a characterization for the inclusion relation is presented in \cite[Th. 2.6]{eek2}. However, we have been unable to find in the literature any explicit characterization of the inclusion relation between bundle closures of matrix pencils under strict equivalence, even though Theorem 3.3 in \cite{eek2} provides an explicit characterization of the so-called ``covering" relation between bundle closures of matrix pencils (where {\em covering} means that $\overline\bun(L_2)\subseteq\overline\bun(L_1)$ and there is no any $L_3$ such that $\overline\bun(L_2)\subset \overline\bun(L_3)\subset\overline\bun(L_1)$, and $\subset$ means strict inclusion).

In the arguments described at the end of the last-but-one paragraph, it is implicitly assumed that the KCF of $L(\la)$ is the ``generic" one in $\overline\bun(L)$. Let us recall that a subset, ${\cal S}_0$, of a given set $\cal S$ (in a topological space) is called {\em generic in $\cal S$} if ${\cal S}_0$ is open and dense in $\cal S$. Therefore, the assumption on the genericity of the KCF just mentioned is equivalent to say that $\bun(L)$ is open and dense in its closure. Clearly, $\bun(L)$ is dense in $\overline\bun(L)$. The property of being open in its closure is well-known for orbits of varieties under the action of a group (see, for instance, \cite[p. 60]{humphreys}). This is the case, for instance, of matrices under similarity and matrix pencils under strict equivalence and congruence. However, we have not found in the literature a corresponding result for bundles so, up to our knowledge, the question on whether the bundles are open in their closure, is still open.

Summarizing, the following two relevant questions arise:
\begin{enumerate}
    \item[Q1.] To provide a characterization for the relation $\overline\bun(L_2)\subseteq\overline\bun(L_1)$ to hold, for two given matrix pencils $L_1$ and $L_2$.
    \item[Q2.] To prove that bundles are open in their closure.
\end{enumerate}

The main goal of the present work is to provide an answer to the previous two questions. To get this answer, we revisit some notions that are already in the literature, like the notion of ``coalescence" of eigenvalues, as well as some results, like Theorem 7.5 in \cite{pervouchine}, which is proven to be false.

As an aside result, we are also able to prove that bundles of matrices under similarity, as well as bundles of matrix polynomials, defined as sets of matrix polynomials with the same size and grade (see Section \ref{matrixbun_sec} for this notion) having the same spectral information, up to the specific values of their eigenvalue (see Section \ref{polybun_sec}), are also open in their closure.

The rest of the paper is organized as follows. In Section \ref{notation_sec} we introduce the notation and basic notions used throughout the paper, together with some already known results that are used later. Section \ref{inclusion_sec} is devoted to the solution of question Q1 above, whereas in Section \ref{open_sec} we provide an affirmative answer to question Q2 for bundles of matrix pencils under strict equivalence, as well as for matrices under similarity (in Section \ref{matrixbun_sec}), and matrix polynomials of higher degree (in Section \ref{polybun_sec}). Finally, Section \ref{conclusion_sec} presents a summary of the main contributions of the paper, together with some lines of further related research.

\section{Basic definitions and notation}\label{notation_sec}

We use the following notation throughout the paper. By $\CC^{m\times n}$ we denote the set of $m\times n$ matrices with complex entries, whereas $\gln$ denotes the set of $n\times n$ invertible matrices with complex entries. Also, $\overline \CC=\CC\cup\{\infty\}$. By $I_k$ we denote the $k\times k$ identity matrix, and $\nu(A)$ denotes the dimension of the (right) null space of the matrix $A$. Instead of using the representation of matrix pencils as pairs of matrices of the same size, we will represent a matrix pencil as $L(\la)=\la B+A$, with $A,B\in\CC^{m\times n}$, namely as a matrix polynomial of degree $1$ in the variable $\la$.  We often denote matrix pencils with a single capital letter (usually $L$ and $M$) and, for the sake of simplicity, in the notation for notions associated with a matrix pencil we omit the variable $\la$, and write just $L$ instead of $L(\la)$.

The {\em rank} of the pencil $L(\la)$ (that is sometimes found in the literature under the name ``normal rank"), denoted by $\rank L$, is the rank of $L(\la)$ considered as a matrix over the field of rational functions in the variable $\la$. In other words, it is the size of the largest non-identically zero minor of $L(\la)$.

Block-partitioned matrices (or matrix pencils) will appear frequently throughout the manuscript, and the blocks indicated with $*$ are not relevant in the arguments, developments, or results. For a block diagonal pencil (or matrix) with diagonal blocks $A_1,\hdots,A_k$ we use either the notation $\diag(A_1,\hdots,A_k)$ or $\bigoplus_{i=1}^kA_i$.

\subsection{The KFC, orbits and bundles}

Let us recall that the KCF of a matrix pencil $L(\la)$ (that we denote by KCF($L$)) is a block diagonal pencil, whose diagonal (``canonical") blocks can be of the following four forms (see, for instance, \cite[Ch. XII, \S5]{gant}):
\begin{itemize}
    \item Jordan blocks associated with a finite eigenvalue $\mu$, namely $\la I_k+J_k(\mu)$, for $k\geq1$, where 
    $$
    J_k(\mu):=\begin{bmatrix}
    -\mu&1\\&\ddots&\ddots\\&&-\mu&1\\&&&-\mu
    \end{bmatrix}_{k\times k}.
    $$
    \item Jordan blocks associated with the infinite eigenvalue, namely $\la N_k+I_k$, for $k\geq1$, where:
    $$
    N_k:=\begin{bmatrix}
    0&1\\&\ddots&\ddots&\\&&0&1\\&&&0
    \end{bmatrix}_{k\times k}.
    $$
    \item Right singular blocks, for $k\geq0$:
    $$
    R_k(\la)=:\begin{bmatrix}
    \la&1\\&\la&1\\&&\ddots&\ddots\\&&&\la&1
    \end{bmatrix}_{k\times (k+1)}.
    $$
\item Left singular blocks, $R_k(\la)^\top$, for $k\geq0$.
\end{itemize}
The KCF of $L(\la)$ is determined up to permutation of the diagonal blocks. 

Let us note that $R_0(\la)$ is a null column, whereas $R_0(\la)^\top$ is a null row. The KCF reveals all the invariants of a matrix pencil under strict equivalence, namely the set of distinct finite and infinite eigenvalues together with the number and the sizes of their associated Jordan blocks, and the number and the sizes of the right and left singular blocks. 

Then, $\mu\in\CC$ is a {\em finite eigenvalue} of $L(\la)$ if KCF($L$) contains, at least, one Jordan block associated with $\mu$, and $L(\la)$ has the {\em infinite eigenvalue} if KCF($L$) contains, at least, one Jordan block associated with the infinite eigenvalue. The pencil $L(\la)$ is said to be {\em regular} if there are neither right nor left singular blocks in KCF($L$) (this is equivalent to say that $L(\la)$ is square and $\det L(\la)$ is a non-identically zero polynomial). By $\Lambda(L)$ we denote the {\em spectrum} of the pencil $L(\la)$ (namely, the set of distinct eigenvalues of $L(\la)$, both finite and infinite).

We denote by $W(\mu,L)=(W_1(\mu,L),W_2(\mu,L),\hdots)$ the {\em Weyr characteristic} of the eigenvalue $\mu$ in the $m\times n$ matrix pencil $L(\la)$. In other words, when $\mu\in\CC$ (respectively, $\mu=\infty$), $W_i(\mu,L)$, for $i\geq1$, is the number of Jordan blocks $\la I_k+J_k(\mu)$ (resp., $\la N_k+I_k$), with $k\geq i$, in KCF($L$), (see, for instance, \cite{dehoyos}). If $\mu$ is not an eigenvalue of $L(\la)$, then $W(\mu,L)=(0,0,\hdots)$. Also, $r(L)=(r_0(L),r_1(L),\hdots)$ and $\ell(L)=(\ell_0(L),\ell_1(L),\hdots)$ denote, respectively, the Weyr characteristic of the right and left singular structure. In other words, $r_i(L)$ (respectively, $\ell_i(L)$) is the number of right (resp., left) singular blocks $R_k(\la)$ of size $k\times (k+1)$ (resp., $R_k(\la)^\top$ of size $(k+1)\times k$), with $k\geq i$, in KCF($L$). In particular, $r_0(L)=n-\rank L$ and $\ell_0(L)=m-\rank L$. Note that $W(\mu,L),r(L),$ and $\ell(L)$ are lists of non-increasing integers. We use the notation ${\cal L}_1\prec{\cal L}_2$ to denote the {\em majorization} of two lists of non-increasing integers, namely: $\sum _{i=1}^j{\cal L}_1(i)\leq \sum_{i=1}^j{\cal L}_2(i)$, for all $j\geq1$, assuming that the lists start with $i=1$.

For $\mu\in\CC$ and $\varepsilon>0$, the $\varepsilon$-neighborhood of $\mu$ is defined as $B(\mu,\varepsilon):=\{z\in\overline\CC:\ |z-\mu|<\varepsilon\}$, whereas if $\mu=\infty$, we set $B(\infty,\varepsilon):=\{z\in\CC:\ |z|>\varepsilon^{-1}\}$).

Following the notation in \cite{pervouchine}, $\Phi$ denotes the set of all one-to-one mappings of $\overline \CC$ to itself. We also use the notation $\Psi$ for the set of all mappings from $\overline\CC$ to itself (not necessarily one-to-one). Then, if $K_L(\la):={\rm KCF}(L)$ (where the canonical blocks are given in any order), and $\psi\in\Psi$, we denote by $\psi(L)$ any pencil which is strictly equivalent to the pencil obtained from $K_L(\la)$ after replacing the Jordan blocks associated with the eigenvalue $\mu\in\overline\CC$ by Jordan blocks of the same size associated with the eigenvalue $\psi(\mu)$, for any eigenvalue $\mu$ of $K_L(\la)$. Note that this pencil $\psi(L)$ is not uniquely determined, but all pencils $\psi(L)$ are strictly equivalent to each other.

For a given matrix pencil $L(\la)=\la B+A$, with $A,B\in\CC^{m\times n}$, we set:
$$
\begin{array}{ccll}
     \orb(L)&:=&\{\la PBQ+PAQ\,:\ P\in\mbox{\rm GL}_m(\CC),Q\in\mbox{GL}_n(\CC)\} &\mbox{(orbit of $L(\la)$),}\\
      \bun(L)&:=&\displaystyle\bigcup_{\varphi\in\Phi}\orb(\varphi(L)) &\mbox{(bundle of $L(\la)$).}
\end{array}
$$
Note that $\orb(\varphi(L))$ is well defined, regardless of the particular pencil $\varphi(L)$, since, as mentioned above, all pencils $\varphi(L)$ are strictly equivalent.

By definition, bundles are the union of orbits of all matrix pencils having the same KCF up to the eigenvalues. Note that this union is infinite provided that the pencils have, at least, one eigenvalue. However, if KCF($L$) contains only blocks of the form $R_k(\la)$ and/or $R_k(\la)^\top$, then $\bun(L)=\orb(L)$. Some special attention should be paid to the infinite eigenvalue. More precisely, the reason for considering maps over $\overline\CC$ is, precisely, to include the infinite eigenvalue in the bundles. For instance, if $L(\la)$ is of the form
$$
L(\la)=\la\left[\begin{array}{cc|c}1&0&0\\0&1&0\\\hline0&0&1\end{array}\right]+\left[\begin{array}{cc|c}-\mu&1&0\\0&-\mu&0\\\hline0&0&-\widetilde\mu\end{array}\right],\qquad \mu\neq\widetilde\mu
$$
(namely, the pencil has two different finite eigenvalues $\mu$ and $\widetilde\mu$ with Jordan blocks of sizes $2$ and $1$, respectively) then
$$
\begin{array}{ccl}
{\cal B}(L)&=&\footnotesize\left\{\la PQ+P\left[\begin{array}{cc|c}-a&1&\\0&-a&\\\hline&&-\widetilde a\end{array}\right]Q:a,\widetilde a\in\CC,a\neq\widetilde a,\ P,Q\in\mbox{GL}_3(\CC)\right\}\\
&&\footnotesize\bigcup\left\{\la P\left[\begin{array}{cc|c}1&0&\\0&1&\\\hline&&0\end{array}\right]Q+P\left[\begin{array}{cc|c}-a&1&\\0&-a&\\\hline&&1\end{array}\right]Q: a\in\CC,\ P,Q\in\mbox{GL}_3(\CC)\right\}\\
&&\footnotesize\bigcup \left\{\la P\left[\begin{array}{cc|c}0&1&\\0&0&\\\hline&&1\end{array}\right]Q+P\left[\begin{array}{cc|c}1&0&\\0&1&\\\hline&&-a\end{array}\right]Q: a\in\CC,\ P,Q\in\mbox{GL}_3(\CC)\right\}.
\end{array}
$$
The first set in the union above corresponds to a direct sum of two Jordan blocks associated with a couple of finite (different) eigenvalues, the second one corresponds to a Jordan block of size $2$ associated with a finite eigenvalue, together with a Jordan block of size $1$ associated with the infinite eigenvalue, and the third set corresponds to a Jordan block of size $2$ associated with the infinite eigenvalue, together with a Jordan block of size $1$ associated with a finite eigenvalue.

We will use the standard notation $\overline S$ for the closure, in the standard topology, of the set $S$. In this context, the set of matrix pencils $\la B+A$, with $A,B\in\CC^{m\times n}$, is identified with $\CC^{2mn}$, and we consider the standard topology in this set. For the closure of the orbit and the bundle of a matrix pencil $L(\la)$ we use the notation $\overline{\orb}(L)$ and $\overline{\bun}(L)$, respectively.

\subsection{Coalescence of eigenvalues}\label{coalescence_sec}
The notion of coalescence of eigenvalues, which is key to describe the inclusion relationships for closures of bundles of matrices and matrix pencils, has been used in previous references, including \cite{eek1,eek2}. We state here a formal definition, which is equivalent to the one mentioned in \cite[Th. 3.3-(5)]{eek2}. We recall that the union of two lists of non-increasing integers (like the Weyr characteristics), say ${\cal L}_1$ and ${\cal L}_2$, that we denote by ${\cal L}_1\cup{\cal L}_2$, consists of a new list of non-increasing integers which is obtained by arranging all numbers in ${\cal L}_1$ and ${\cal L}_2$ in a non-increasing order. For instance, if 
$
{\cal L}_1=(5,2)$ and ${\cal L}_2=(6,3,3,2,1),
$
then ${\cal L}_1\cup{\cal L}_2=(6,5,3,3,2,2,1)$.

To understand the following definition, we recall that $\Psi$ denotes the set of mappings from $\overline\CC$ to itself.

\begin{definition}\label{coalescence_def} {\rm (Coalescence of eigenvalues).}
Let $L(\la)$ be a matrix pencil with distinct eigenvalues $\mu_1,\hdots,\mu_s$, and let $\psi\in\Psi$. Then, $\psi_c(L)$ is any matrix pencil of the same size as $L(\la)$ satisfying the following three properties:
\begin{itemize}
    \item $r(\psi_c(L))=r(L)$,
    \item $\ell(\psi_c(L))=\ell(L)$, and
    \item $W(\mu,\psi_c(L))=\bigcup_{\mu_i\in\psi^{-1}(\mu)}W(\mu_i,L)$, for all $\mu\in\overline{\CC}$.
\end{itemize}
We say that the eigenvalues $\mu_{i_1},\hdots,\mu_{i_d}$ of $L(\la)$ have {\em coalesced} to the eigenvalue $\mu$ in $\psi_c(L)$ if $\psi^{-1}(\mu)=\{\mu_{i_1},\hdots,\mu_{i_d}\}\cup S$, with $S\cap\Lambda(L)=\emptyset$.
\end{definition}

\begin{remark}
The matrix pencil $\psi_c(L)$ in Definition \ref{coalescence_def} is not uniquely defined, but all pencils $\psi_c(L)$, for some given $L(\la)$ and $\psi$, are strictly equivalent to each other, since they all have the same KCF  (as it happens with $\psi(L)$). Moreover, note that $\bun(\psi_c(L))=\bun(\widetilde \psi_c(L))$, for any $\psi,\widetilde\psi\in\Psi$ such that $\psi(\mu)=\psi(\widetilde\mu)$ if and only if $\widetilde\psi(\mu)=\widetilde\psi(\widetilde\mu)$, for any $\mu\neq\widetilde\mu$.
\end{remark}

Coalescence of eigenvalues is a way of ``gathering" eigenvalues by taking the union of their Weyr characteristics. The following example aims to illustrate this notion.

\begin{example}
Let $L(\la)$ be the following pencil, already given in KCF:
$$
\begin{array}{rl}
L(\la)=\diag&(R_3(\la),R_1(\la),\la I_2+J_2(0),\la I_2+J_2(0),\la I_1+J_1(0),\\
&\la I_3+J_3(1),\la I_2+J_2(1),\la I_4+J_4(2),R_2(\la)^\top),
\end{array}
$$
so that $r(L)=(2,2,1,1),$ $\ell(L)=(1,1,1),$ and $W(0,L)=(3,2),$ $W(1,L)=(2,2,1),$ $W(2,L)=(1,1,1,1).$

Let $\psi:\overline{\CC}\rightarrow\overline\CC$ be such that $\psi(0)=\psi(1)=\psi(2)=1$. Then any pencil $\psi_c(L)$ is of the form:
$$
\psi_c(L)=P\cdot\diag(R_3(\la),R_1(\la),\la I_9+J_9(1),\la I_4+J_4(1),\la I_1+J_1(1),R_2(\la)^\top)\cdot Q,
$$
for some invertible matrices $P,Q$. Note, indeed, that, for any $\psi_c(L)$ as above, $r(\psi_c(L))=(2,2,1,1)=r(L)$, $\ell(\psi_c(L))=(1,1,1)=\ell(L)$, and $W(1,\psi_c(L))=(3,2,2,2,1,1,1,1,1)=W(0,L)\cup W(1,L)\cup W(2,L)$.

However, if $\psi$ is such that $\psi(0)=\psi(2)=1,\psi(1)=5$, then any pencil $\psi_c(L)$ is now of the form
$$
\begin{array}{cl}
\psi_c(L)=&P\cdot\diag(R_3(\la),R_1(\la),\la I_6+J_6(1),\la I_2+J_2(1),\la I_1+J_1(1),\\&\la I_3+J_3(5),\la I_2+J_2(5),R_2(\la)^\top)\cdot Q,
\end{array}
$$
for some invertible matrices $P,Q$. Note, indeed, that, for any $\psi_c(L)$ as above, $r(\psi_c(L))=(2,2,1,1)=r(L)$, $\ell(\psi_c(L))=(1,1,1)=\ell(L)$, and 
$W(1,\psi_c(L))=(3,2,1,1,1,1)=W(0,L)\cup W(2,L)$, $W(5,\psi_c(L))=(2,2,1)=W(1,L)$.
\end{example}

\subsection{Some basic results}\label{basicresults_sec}

We are interested in the majorization of the Weyr characteristics of an eigenvalue in two given matrix pencils. This majorization is defined by inequalities on the sum of the first elements of the corresponding Weyr characteristic, so a formula for this sum would be quite useful to this end. Such a formula will come from the rank of certain big block-partitioned matrices. For describing this, we follow  some of the developments in the PhD. thesis \cite{tesis-inma}, that we include here for the sake of completeness.

Given the matrix pencil $L(\la)=\la B+A$, with $A,B\in\CC^{m\times n}$ and $\mu\in\CC$, we define the following block-partitioned matrices with $d$ block columns and $d$ block rows:
\begin{equation}\label{pmu}
P_\mu^d(L):=\begin{bmatrix}
L(\mu)&0&\hdots&\hdots&0\\
B&L(\mu)&\ddots&&\vdots&\\
0&B&\ddots&\ddots&\vdots&\\
\vdots&\ddots&\ddots&L(\mu)&0\\
0&\hdots&0&B&L(\mu)
\end{bmatrix}_{dm\times dn},\qquad\mbox{for $d\geq1$}.
\end{equation}
Note that 
$$
P_\mu^d(L)=\begin{bmatrix}0&0\\I_{d-1}&0\end{bmatrix}\otimes B+I_d\otimes L(\mu),
$$
where $\otimes$ denotes the Kronecker product.

For $d_1,\hdots,d_s\geq1$ and $\la_1,\hdots,\la_s\in\CC$ (different to each other), we define the following block-partitioned matrices:
$$
P_{\la_1,\hdots,\la_s}^{d_1,\hdots,d_s}(L):=\begin{bmatrix}
P_{\la_1}^{d_1}(L)&0&0&\hdots&0\\
Q_{d_2,d_1}(B)&P_{\la_2}^{d_2}(L)&0&\hdots&0\\
\vdots&\ddots&\ddots&&\vdots\\
0&0&\hdots&Q_{d_s,d_{s-1}}(B)&P_{\la_s}^{d_s}(L)
\end{bmatrix},
$$
where, for $2\leq i\leq s$, 
$$
Q_{d_i,d_{i-1}}(B):=\begin{bmatrix}
0&\hdots&0&B\\
0&\hdots&0&0\\
\vdots&\ddots&\vdots&\vdots&\\
0&\hdots&0&0
\end{bmatrix}\in\CC^{d_i m\times d_{i-1}n}.
$$

\begin{lemma}\label{equivalence_lemma}
Let $L(\la)=\la B+A$. Then, for any distinct $\la_1,\hdots,\la_s\in\CC$ and $d_1,\hdots,d_s\geq1$,
$$
\nu(\diag(P_{\la_{1}}^{d_1}(L),\hdots,P_{\la_{s}}^{d_s}(L)))=\nu(P_{\la_{1},\hdots,\la_{s}}^{d_1,\hdots,d_s}(L)),
$$
where $\nu(Z)$ denotes the dimension of the right nullspace of the matrix $Z$.
\end{lemma}
\begin{proof}
The result is a consequence of the fact that  $P_{\la_{1},\hdots,\la_{s}}^{d_1,\hdots,d_s}(L)$ is equivalent, by elementary block-row and block-column operations, to the block diagonal matrix $\diag(P_{\la_{1}}^{d_1}(L),\hdots,P_{\la_{s}}^{d_s}(L))$ (because these operations preserve the rank and, as a consequence, the dimension of the null space, of the matrix). To see this equivalence, let us first consider just two big blocks, namely $s=2$, so the matrix reads $P_{\alpha,\beta}^{d_1,d_2}(L)$, with $\alpha,\beta\in\CC$ and $\alpha\neq\beta$. Note that, for any $s\in\CC$, 
\begin{equation}\label{sb} 
sB=s\left(\frac{1}{\alpha-\beta}L(\alpha)-\frac{1}{\alpha-\beta}L(\beta)\right).
\end{equation}
Therefore, adding to the $d_1$th block-column of $P_{\alpha,\beta}^{d_1,d_2}(L)$ the $(d_1+1)$st block column multiplied by $1/(\alpha-\beta)$, and to the $(d_1+1)$st block-row the $d_1$th one multiplied by $-1/(\alpha-\beta)$, the resulting matrix has a $0$ block in the $(d_1+1,d_1)$ block position, instead of the block $B$ that appears in $P_{\alpha,\beta}^{d_1,d_2}(L)$, but the blocks in the positions $(d_1+1,d_1-1)$ and $(d_1+2,d_1)$ are of the form $s_1B$ and $s_2B$, for some $s_1,s_2\in\CC$. These blocks are in a block diagonal that is below the one containing $B$ in the original matrix $P_{\alpha,\beta}^{d_1,d_2}(L)$, so the previous block row and block column operations have taken the block $B$ to some multiples of $B$ in the lower block diagonal. Because of \eqref{sb}, we can find appropriate elementary block row and block column operations that turn these blocks into $0$, but creates some nonzero blocks, which are, again, multiples of $B$, in the next lower block diagonal. Proceeding recursively in this way, we end up with a matrix which is equal to $\diag(P_\alpha^{d_1}(L),P_\beta^{d_2}(L)$ except for a nonzero block of the form $sB$, for some $s\in\CC$, in the $(d_1+d_2,1)$ block-position. Using again \eqref{sb}, an appropriate linear combination of the first block row and the last (i. e., the $(d_1+d_2)$th one) block column will shrink this block to $0$, so we end up with $\diag(P_\alpha^{d_1}(L),P_\beta^{d_2}(L))$, as wanted.

 If there are more than two blocks, we partition the matrix $P_{\la_{1},\hdots,\la_{s}}^{d_1,\hdots,d_s}(L)$ into a $2\times2$ block matrix, as follows:
\begin{equation}\label{pl-partition}
P_{\la_{1},\hdots,\la_{s}}^{d_1,\hdots,d_s}(L)=\begin{bmatrix}P_{\la_{1},\hdots,\la_{s-1}}^{d_1,\hdots,d_{s-1}}(L)&0\\Q_{d_s,\widetilde d_{s-1}}(B)&P_{\la_s}^{d_s}(L)\end{bmatrix},
\end{equation}
with $\widetilde d_{s-1}:=d_1+d_2+\cdots+d_{s-1}$. Using appropriate elementary block-row and block-column operations as explained before for the case $s=2$, this matrix is equivalent to 
\begin{equation}\label{pl-0block}
\begin{bmatrix}P_{\la_{1},\hdots,\la_{s-1}}^{d_1,\hdots,d_{s-1}}(L)&0\\0&P_{\la_s}^{d_s}(L)\end{bmatrix}.
\end{equation}
To be more precise, in this case, instead of \eqref{sb}, we need to use the identity
$$
sB=s\left(\frac{1}{\la_i-\la_s} L(\la_i)-\frac{1}{\la_i-\la_s} L(\la_s)\right),
$$
for $i=1,\hdots,s-1$, and the block-row and block-column operations produce some multiples of $B$ in the $(2,1)$ block of \eqref{pl-partition}, that, following the same procedure as before, will move from one block diagonal to the ``lower" one. Then, after a finite number of elementary block-row and block-column operations we arrive at \eqref{pl-0block}. From \eqref{pl-0block} we can proceed in the same way with the upper left block $P_{\la_{1},\hdots,\la_{s-1}}^{d_1,\hdots,d_{s-1}}(L)$, and so on, until we get $\diag(P_{\la_{1}}^{d_1}(L),\hdots,P_{\la_{s}}^{d_s}(L))$.
\end{proof}

The following result provides a formula for the sum of the first $d$ terms in the Weyr characteristic of an eigenvalue of a matrix pencil, in terms of the dimension of the null space of the matrix in \eqref{pmu}.

\begin{lemma}\label{rank_lemma}
Let $L(\la)$ be an $m\times n$ matrix pencil and $\mu\in\CC$. Then, for all $d\geq1$:
$$
\nu(P_\mu^d(L))=\sum_{i=1}^dW_i(\mu,L)+d\,r_0(L),
$$
where $r_0(L)=n-\rank L$.
\end{lemma}
\begin{proof}
It is straightforward to see that, if $\widetilde L(\la)$ is a strictly equivalent pencil to $L(\la)$, then $\nu(P_\mu^d(L))=\nu(P_\mu^d(\widetilde L))$. Therefore, we may assume that $L(\la)$ is given in KCF. Let us decompose $L(\la)=\diag(L_r(\la),L_{\rm reg}(\la),L_\ell(\la))$, where $L_{\rm reg}(\la)$ contains all Jordan blocks of $L(\la)$, $L_r(\la)$ contains all blocks of the form $R_k(\la)$, whereas $L_\ell(\la)$ contains the blocks of the form $R_k(\la)^\top$. By means of column and row permutations, $P_\mu^d(\diag(L_r,L_{\rm reg},L_\ell))$ is strictly equivalent to $\diag(P_\mu^d(L_r),P_\mu^d(L_{\rm reg})),P_\mu^d(L_\ell))$, so $\nu(P_\mu^d(L))=\nu(P_\mu^d(\diag(L_r,L_{\rm reg},L_\ell)))=\nu(\diag(P_\mu^d(L_r),P_\mu^d(L_{\rm reg}),P_\mu^d(L_\ell)))=\nu(P_\mu^d(L_r))+\nu(P_\mu^d(L_{\rm reg}))+\nu(P_\mu^d(L_\ell))$.

Now, since $P_\mu^d(L_\ell)$ has full column rank, for every $\mu\in\CC$, then $\nu(P_\mu^d(L_\ell))=0$. Also, since $r_0(L)$ is the number of blocks $R_k(\la)$ in $L_r(\la)$, it is also straightforward to see that $\nu(P_\mu^d(L_r))=d\,r_0(L)$. Finally, $\nu(P_\mu^d(L_{\rm reg}))=\sum_{i=1}^d W_i(\mu,L_{\rm reg})=\sum_{i=1}^d W_i(\mu,L)$, where the first identity is also straightforward to get (see also \cite[p. 36]{glr}).
\end{proof}

In the case where $r_0=0$, Lemma \ref{rank_lemma} is a consequence of the developments carried out in \cite[\S 5]{kk86}.

The Weyr characteristic of the eigenvalue $\mu$ in the pencil $L(\la)$ is the {\em conjugate} partition of the {\em Segre characteristic}, denoted by $S(\mu,L)=(S_1(\mu,L),S_2(\mu,L),\hdots)$, where $S_i(\mu,L)$ is the size of the $i$th largest Jordan block associated with the eigenvalue $\mu$ in KCF($L$) (see, for instance, \cite{dehoyos}). In general, the conjugate of a partition, ${\bf a}$, is the partition, denoted by ${\bf a}^\sharp$, whose $i$th element, for $i\geq1$, is equal to the number of elements in ${\bf a}$ which are greater than or equal to $i$. Also, the sum of a finite number of partitions is the partition whose $i$th element is the sum of the $i$th elements in all partitions (adding zeroes at the end of the partitions if necessary). The following result can be found in \cite[p. 6]{macdonald}. 

\begin{lemma}\label{dual_lem}
Let ${\bf a}_1,\hdots,{\bf a}_k$ be partitions. Then
$$
\left(\bigcup_{i=1}^k{\bf a}_i\right)^\sharp=\sum_{i=1}^k{\bf a}_i^\sharp,\qquad\mbox{and}\qquad
\left(\sum_{i=1}^k{\bf a}_i\right)^\sharp=\bigcup_{i=1}^k{\bf a}_i^\sharp.
$$
\end{lemma}

The characterization for the inclusion between orbit closures of matrix pencils, which was obtained independently in \cite{bongartz} and \cite{pokrzywa}, and reformulated in \cite{dehoyos}, will be used several times in the paper. We include it here for completeness (the statement we present is in between the ones of \cite[Lemma 1.3]{dehoyos} and \cite[Th. 3.1]{eek2}).

\begin{theorem}\label{orbitinclusion_th}{\rm (Characterization of the inclusion between orbit closures).} The matrix pencil $P_2(\la)$ belongs to the closure of the orbit of the pencil $P_1(\la)$ (in other words, $\overline\orb(P_2)\subseteq\overline\orb(P_1) $) if and only if the following relations hold:
\begin{itemize}
    \item[{\rm (i)}] $r(P_2)\prec r(P_1)+(h,h,\hdots),$
    \item[{\rm (ii)}] $\ell(P_2)\prec \ell(P_1)+(h,h,\hdots),$
    \item[{\rm (iii)}] $W(\mu,P_1)\prec W(\mu,P_2)+(h,h,\hdots),$ for all $\mu\in\overline\CC$,
\end{itemize}
where $h:=\rank P_1-\rank P_2$.
\end{theorem}

One of the main goals of this paper is to provide a characterization like the one in Theorem \ref{orbitinclusion_th} for closures of bundles instead of orbits. This will be provided in Theorem \ref{bundom_th}.

\section{A characterization for the inclusion of bundle closures of matrix pencils}\label{inclusion_sec}

The main result of this section is Theorem \ref{bundom_th}, which provides a characterization for the inclusion of bundle closures of matrix pencils. To prove it we use the following technical result. It is a consequence of \cite[Teorema 2.3]{tesis-inma}, adapted to the conditions in the statement. In the proof we use the notion of {\em reversal} of the matrix pencil $L(\la)=\la B+A$, defined as ${\rm rev}L:=\la A+B$.

\begin{theorem}\label{coalescence_th}
Let $\{M_k(\la)\}_{k\in\mathbb N}$ be a sequence of $m\times n$ complex matrix pencils such that:
\begin{enumerate}[{\rm (i)}]
\item $M_k(\la)\in\bun(L)$, for some $m\times n$ complex matrix pencil $L(\la)$ and for all $k\in\mathbb N$,
    \item\label{equalweyr} $\la_{1,k},\hdots,\la_{s,k}\in\overline\CC$ are distinct eigenvalues of $M_k(\la)$, for $k\in\mathbb N$, and $W(\la_{i,k_1},M_{k_1})=W(\la_{i,k_2},M_{k_2})$, for all $i=1,\hdots,s$, and all $k_1,k_2\in\mathbb N$, 
    \item\label{convergence} $\{M_k(\la)\}_{k\in\mathbb N}$ converges to $M(\la)$, and
    \item the sequence $\{\la_{i,k}\}_{k\in\mathbb N}$ converges to $\mu$, for all $i=1,\hdots,s$, where $\mu\in\overline\CC$.
\end{enumerate}
Then $\bigcup_{i=1}^sW(\la_{i,k},M_k)\prec W(\mu,M)+(h,h,\hdots)$, where $h:=\rank L-\rank M$.
\end{theorem}
\begin{proof}
Let us first assume that $\mu\in\CC$ (namely, $\mu\neq\infty$). Then, for $k$ large enough, $\la_{i,k}\in\CC$, for all $1\leq i\leq s$. Therefore, we may assume that $\la_{i,k}\in\CC$, for all $1\leq i\leq s$ and all $k\in\mathbb N$. In the rest of this case, we essentially follow the proof of \cite[Teorema 2.3]{tesis-inma}. 

Set $m=(m_1,m_2,\hdots):=\bigcup_{i=1}^sW(\la_{i,k},M_k)$ (note that this does not depend on $k$, by condition \eqref{equalweyr} in the statement). By definition of union of Weyr characteristics, for each $d\geq 1$, there are $d_1,\hdots,d_s\geq0$ such that $d_1+\cdots+d_s=d$ and
\begin{equation}\label{sumi}
\sum_{i=1}^dm_i=\sum_{j=1}^{d_1}W_j(\la_{1,k},M_k)+\cdots+\sum_{j=1}^{d_s}W_j(\la_{s,k},M_k),
\end{equation}
for all $k\in\mathbb N$, where $\sum_{j=1}^0W_j(\la_{i,k},M_k):=0$, for $i=1,\hdots,s$.

Since the pencils $M_k(\la)$ converge to $M(\la)$ and the values $\la_{i,k}$ converge to $\mu$, for all $1\leq i\leq s$, taking into account that $d_1+\cdots+d_s=d$, we conclude that the matrices $P_{\la_{1,k},\hdots,\la_{s,k}}^{d_1,\hdots,d_s}(M_k)$ (where, if $d_i=0$ for some $1\leq i\leq s$, the block row and block column corresponding to $P_{\la_{i,k}}^{d_i}(M_k)$ is not present) converge to  $P_\mu^d(M)$. Then, by the lower semi-continuity of the rank, we get
\begin{equation}\label{rankineq}
\rank P_\mu^d(M)\leq \rank P_{\la_{1,k},\hdots,\la_{s,k}}^{d_1,\hdots,d_s}(M_k),
\end{equation}
for $k$ large enough.

From Lemma \ref{rank_lemma}, and taking into account that $\rank M_k=\rank L$, for all $k\in\mathbb N$, we get, for each $1\leq i\leq s$ and all $k\in\mathbb N$:
$$
\sum_{j=1}^{d_i}W_j(\la_{i,k},M_k)+d_i(n-\rank L)=\nu(P_{\la_{i,k}}^{d_i}(M_k))=nd_i-\rank(P_{\la_{i,k}}^{d_i}(M_k)).
$$
From this identity and \eqref{sumi} we obtain, for all $k\in\mathbb N$,
\begin{equation}\label{nubigp}
\begin{array}{ccl}
\displaystyle\sum_{i=1}^d m_i+d(n-\rank L)&=&\displaystyle\sum_{i=1}^s\left(\sum_{j=1}^{d_i}W_j(\la_{i,k},M_k)+d_i(n-\rank L)\right)\\
&=&\displaystyle nd-\sum_{i=1}^s\rank(P_{\la_{i,k}}^{d_i}(M_k))\\
&=&nd-\rank\diag(P_{\la_{1,k}}^{d_1}(M_k),\hdots,P_{\la_{s,k}}^{d_s}(M_k))
\\&=&\nu(\diag(P_{\la_{1,k}}^{d_1}(M_k),\hdots,P_{\la_{s,k}}^{d_s}(M_k)))\\
&=&\nu(P_{\la_{1,k},\hdots,\la_{s,k}}^{d_1,\hdots,d_s}(M_k)),
\end{array}
\end{equation}
where the last identity is a consequence of Lemma \ref{equivalence_lemma}. 

Now \eqref{nubigp}, together with \eqref{rankineq} and Lemma \ref{rank_lemma} imply that 
$$
\begin{array}{ccl}
\displaystyle\sum_{i=1}^d m_i+d(n-\rank L)&=&\nu(P_{\la_{1,k},\hdots,\la_{s,k}}^{d_1,\hdots,d_s}(M_k))\\&\leq&\displaystyle
\nu(P_\mu^d(M))=\sum_{i=1}^dW_i(\mu,M)+d(n-\rank M),
\end{array}
$$
or, equivalently,
$$
\sum_{i=1}^d m_i\leq \sum_{i=1}^dW_i(\mu,M)+d(\rank L-\rank M),
$$
as wanted. 

Now, let us assume that $\mu=\infty$. It is straightforward to see (but we refer otherwise to \cite[Remark 4.3]{4m-mobius}) that, for a given pencil $L(\la)$ and $\mu\in\overline\CC$, the identity $W(\mu,L)=W(\mu^{-1},{\rm rev}L)$ holds. It is also immediate to see that $\rank({\rm rev} L)=\rank L$. Moreover, if $\{M_k(\la)\}_{k\in\mathbb Z}$ is a sequence of pencils converging to $M(\la)$, then $\{{\rm rev}M_k\}_{k\in\mathbb Z}$ converges to ${\rm rev}M$. So let $L(\la), M(\la)$, and $M_k(\la)$, as well as $\la_{i,k}$, be as in the statement. Then $\la_{i,k}^{-1}$ are eigenvalues of ${\rm rev} M_k$, and they converge to the eigenvalue $0$ of ${\rm rev}M$. By the case just proved for $\mu\in\CC$, we conclude that 
$$
\bigcup_{i=1}^sW(\la_{i,k}^{-1},{\rm rev}M_k)\prec W(0,{\rm rev} M)+(h,h,\hdots).
$$
But, since $W(\la_{i,k}^{-1},{\rm rev}M_k)=W(\la_{i,k},M_k)$ and $W(0,{\rm rev}M)=W(\infty,M)$, this implies
$$
\bigcup_{i=1}^sW(\la_{i,k},M_k)\prec W(\infty, M)+(h,h,\hdots),
$$
as claimed.
\end{proof}

The following result provides a theoretical answer to the question of deciding whether a given matrix pencil belongs to the closure of the bundle of another pencil or not. 

\begin{theorem}\label{bundomprev_th}
Let $L(\la)$ and $M(\la)$ be two complex matrix pencils of the same size. Then $M(\la)\in\overline\bun(L)$ if and only if $M(\la)\in\overline\orb(\psi_c(L))$, for some map $\psi:\overline\CC\rightarrow\overline\CC$.
\end{theorem}
\begin{proof}
Let us first prove the ``if" part of the statement. Assume that $M(\la)\in\overline\bun(L)$. Then, there is a sequence of pencils $\{M_k(\la)\}_{k\in\mathbb N}$ such that $M_k(\la)\in\bun(L)$ and $\{M_k(\la)\}_{k\in \mathbb N}$ converges to $M(\la)$. Since $M_k(\la)\in\bun(L)$, we conclude:
\begin{itemize}
    \item $r(M_k)=r(L)$ and $\ell(M_k)=\ell(L)$, for all $k\in\mathbb N$.
    \item All $M_k(\la)$ have the same number of distinct eigenvalues, say $s$, which is the number of distinct eigenvalues of $L(\la)$. Moreover, if $\la_1,\hdots,\la_s$ and $\lambda_{1,k},\hdots,\lambda_{s,k}$ denote, respectively, the distinct eigenvalues of $L(\la)$ and $M_k(\la)$, then $W(\lambda_{i,k},M_k)=W(\la_i,L)$, for all $k\in\mathbb N$ and all $i=1,\hdots,s$. 
\end{itemize}
If we set $h:=\rank L-\rank M$ and $\sigma_\varepsilon(M)=\bigcup_{\mu\in\Lambda(M)}B(\mu,\varepsilon)$ 
then, for all $k$ large enough, the pencils $M_k(\la)$ satisfy (see, for instance, Lemma 1.1 and its proof in \cite{dehoyos}):
\begin{itemize}
    \item[(i)] $r(M)\prec r(M_k)+(h,h,\hdots)=r(L)+(h,h,\hdots)$,
    \item[(ii)] $\ell(M)\prec \ell(M_k)+(h,h,\hdots)=\ell(L)+(h,h,\hdots)$,   
    \item[(iii)] $W(\mu,M_k)\prec W(\mu,M)+(h,h,\hdots)$, for all $\mu\in(\overline\CC-\sigma_\varepsilon(M))\cup \Lambda(M)$.
\end{itemize}
 Now, we are going to see that, by taking a subsequence of $\{M_k(\la)\}_{k\in\mathbb N}$ if necessary, we may assume that, for any $i=1,\hdots,s$ one of the following conditions holds:
\begin{enumerate}[{\rm (C1)}]
    \item\label{convergent} $\{\lambda_{i,k}\}$ converges to an eigenvalue of $M$, or
    \item\label{neighborhood} there is some $\varepsilon>0$ such that $\la_{i,k}\in\overline\CC- \sigma_\varepsilon(M)$, for all $k$ large enough.
\end{enumerate}
Assume that, for some $i=1,\hdots,s$, condition (C\ref{neighborhood}) does not hold. Then, for any $\varepsilon>0$ there is an infinite subsequence of $\{\la_{i,k}\}$ included in $\sigma_\varepsilon(M)$. Since the spectrum of $M$ is finite, there is a subsequence, $\{\lambda_{i,k_j}\}$ that converges to some eigenvalue of $M(\la)$, so $\{\lambda_{i,k_j}\}$ satisfies (C\ref{convergent}). Since $s$ is a finite number, we can keep going with this procedure, by taking a finer subsequence if necessary, until we end up with a subsequence of $\{M_k(\la)\}$, 
whose eigenvalues $\{\la_{i,k}\}$ satisfy either (C\ref{convergent}) or (C\ref{neighborhood}), for all $i=1,\hdots,s$.

Now, let us assume that the eigenvalues $\lambda_{i,k}$ converge to some distinct $\mu_1,\hdots,\mu_d$, for all $i=1,\hdots,t$ and $t\leq s$, where $\mu_1,\hdots,\mu_d$ are eigenvalues of $M(\la)$. Note, however, that not all eigenvalues of $M_k(\la)$ must converge to eigenvalues of $M(\la)$, since some of them can satisfy condition (C\ref{neighborhood}) above, and that not every eigenvalue of $M(\la)$ is necessarily obtained in this way, because some new eigenvalues not coming from eigenvalues of $M_k(\la)$ could arise in the limit $M(\la)$. In particular, we are assuming that the eigenvalues $\la_{i,k}$, for $i=t+1,\hdots,s$, satisfy (C\ref{neighborhood}) above. Let us gather all eigenvalues $\lambda_{i,k}$ that converge to the same eigenvalue $\mu_j$ of $M(\la)$ in the following way: we decompose the set of indices $\{1,\hdots,t\}$ as the union of $d$ disjoint sets of indices, denoted by $I_1,\hdots, I_d$ in such a way that the eigenvalues $\lambda_{i,k}$ with $i\in I_j$ converge to $\mu_j$, for $j=1,\hdots,d$.
This means, in particular, that $M(\la)$ has, at least, $d$ different eigenvalues, namely $\mu_1,\hdots,\mu_d$.

Let $\widetilde L:=\psi_c(L)$, where $\psi:\overline{\CC}\rightarrow\overline{\CC}$ is such that $\psi(\la_i)=\mu_j$, for $i\in I_j$ and $j=1,\hdots,d$, $\psi(\la_i)=\la_i$, for $i=t+1,\hdots,s$, and $\psi(\mu)\not\in\{\mu_1,\hdots,\mu_d,\la_{t+1},\hdots,\la_s\}$, for $\mu\not\in\{\la_1,\hdots,\la_s\}$. Then, $\mu_1,\hdots,\mu_d,\la_{t+1},\hdots,\la_s$ are the distinct eigenvalues of $\widetilde L$, and:
\begin{itemize}
    \item[(iv)] $r(M_k)=r(L)=r(\widetilde L)$,
    \item[(v)] $\ell(M_k)=\ell(L)=\ell(\widetilde L)$,
    \item[(vi)] $W(\mu,\widetilde L)\prec W(\mu,M)+(h,h,\hdots)$, for all $\mu\in\overline\CC$.
\end{itemize}
Claim (vi) is a consequence of the following facts:
\begin{itemize}
    \item If $\mu$ is not an eigenvalue of $\widetilde L(\la)$, then $(0)=W(\mu,\widetilde L)\prec W(\mu,M)+(h,h,\hdots)$, since $h\geq0$.
    \item $W(\mu_j,\widetilde L)=\bigcup_{i\in I_j}W(\la_i,L)=\bigcup_{i\in I_j}W(\la_{i,k},M_k)\prec W(\mu_j,M)+(h,h,\hdots)$, for $j=1,\hdots,d$, where the first identity is a consequence of the definition of $\psi_c(L)$, and the majorization is a consequence of Theorem \ref{coalescence_th}.
    \item For the remaining eigenvalues of $\widetilde L$, namely $\la_i$ with $i=t+1,\hdots,s$, we have $W(\la_i,\widetilde L)=W(\la_i,L)=W(\la_{i,k},M_k)\prec W(\la_{i,k},M)+(h,h,\hdots)\break=(0)+(h,h,\hdots)\prec W(\la_i,M)+(h,h,\hdots)$, where the first majorization is a consequence of (iii) (which applies because $\la_{i,k}$ satisfies (C\ref{neighborhood})), the subsequent equality is a consequence of (C\ref{neighborhood}), which implies that $\la_{i,k}$ is not an eigenvalue of $M(\la)$, and the last majorization is immediate. 
\end{itemize}
Now, (i), (ii) together with (iv)--(vi) imply:
\begin{itemize}
    \item[(a)] $r(M)\prec r(\widetilde L)+(h,h,\hdots)$,
    \item[(b)] $\ell(M)\prec \ell(\widetilde L)+(h,h,\hdots)$,
    \item[(c)] $W(\mu,\widetilde L)\prec W(\mu ,M)+(h,h,\hdots)$, for all $\mu\in\overline\CC$.
\end{itemize}
But (a)--(c) in turn imply that $M(\la)\in\overline\orb(\widetilde L)$, according to Theorem \ref{orbitinclusion_th}, and this proves the ``if" part of the statement.

Now, let us prove the ``only if" part, so let $L(\la)$ be given. We first prove that $\overline\bun(L)$ contains $\overline\orb(\psi_c(L))$, for all maps $\psi:\overline\CC\rightarrow\overline\CC$. 

We are going to first prove that  $\psi_c( L)\in\overline\bun(L)$. To this end, note that $\bun(L)=\bun(P\cdot L(\la)\cdot Q)$, for any invertible matrices $P,Q$ of appropriate size. Therefore, we can assume that $L(\la)$ is given in KCF so, without loss of generality, $L(\la)$ is of the form 
$$
L(\la)=\diag(J_{\lambda_1}(\lambda),\hdots,J_{\lambda_s}(\la),\widehat L(\la)),
$$ 
where 
$$
J_{\lambda_i}(\la):=\displaystyle\bigoplus _{j=1}^{g_i} \left(\la I_{\alpha_{j,i}}+J_{\alpha_{j,i}}(\lambda_i)\right),\qquad\mbox{for $i=1,\hdots,s$,}
$$
with $\alpha_{1,i}\geq\cdots\geq\alpha_{g_i,i}$, for $i=1,\hdots,s$, and  $\Lambda(\widehat L)\cap\{\la_1,\hdots,\la_s\}=\emptyset$. Unlike what happens in the first part of the proof (the ``if" part), here $\la_1,\hdots,\la_s$ denote some distinct eigenvalues, but not necessarily all the eigenvalues of $L(\la)$. Also, for simplicity we are assuming that $\la_1,\hdots,\la_s$ are all finite eigenvalues. For the case where some $\la_k$, for $1\leq k\leq s$, is the infinite eigenvalue, the arguments are also valid just replacing $\la I_{\alpha_{j,k}}+J_{\alpha_{j,k}}(\la_k)$ by $I_{\alpha_{j,k}}+\la N_{\alpha_{j,k}}$, and incorporating the appropriate changes in the corresponding blocks of the forthcoming arguments.

Let us first prove that, if $\psi\in\Psi$ is such that $\psi^{-1}(\mu)=\{\la_1,\hdots,\la_s\}\cup S$, with $S\cap\Lambda(L)=\emptyset$, for some $\mu\in\overline{\CC}$, then $\psi_c(L)\in\overline\bun(L)$.

By definition of $\psi_c(L)$, it suffices to find some $\widetilde L(\la)$ such that
$\widetilde L(\la)\in \overline\bun(L)$ and ${\rm KCF}(\widetilde L)=\diag (\widetilde J_{\mu}(\la),\widehat L(\la))$, where $\mu$ is the only eigenvalue of $\widetilde J_{\mu}(\la)$ (and is not an eigenvalue of $\widehat L(\la)$), and such that $W(\mu,\widetilde L)=\bigcup_{i=1}^s W(\la_i,L)$. 

Without loss of generality, we may assume that $g_1\geq\cdots\geq g_s$. For brevity, we denote by $E_m(i_1,\hdots,i_t)$ the $m\times m$ matrix having an entry equal to $1$ in the positions $(i_1,i_1+1),\hdots,(i_t,i_t+1)$, and zeroes elsewhere. Then, set
$$
\begin{array}{rcl}
L_k(\la)&:=&\diag\left(\displaystyle\bigoplus_{j=1}^{g_s}\left(\bigoplus_{i=1}^s \left(\la I_{\alpha_{j,i}}+J_{\alpha_{j,i}}\left(\mu+\frac{i}{k}\right)\right)\right.+\right.\\&& \left. E_{\alpha_{j,1}+\cdots+\alpha_{j,s}}(\alpha_{j,1},\alpha_{j,1}+\alpha_{j,2},\hdots,\alpha_{j,1}+\cdots+\alpha_{j,s-1})\right),\\
&&\displaystyle\bigoplus_{j=g_s+1}^{g_{s-1}}\left(\bigoplus_{i=1}^{s-1} \left(\la I_{\alpha_{j,i}}+J_{\alpha_{j,i}}\left(\mu+\frac{i}{k}\right)\right)\right.+\\&&\left.  E_{\alpha_{j,1}+\cdots+\alpha_{j,s-1}}(\alpha_{j,1},\alpha_{j,1}+\alpha_{j,2},\hdots,\alpha_{j,1}+\cdots+\alpha_{j,s-2})\right),\\
&&\left.,\hdots,\displaystyle\bigoplus_{j=g_2+1}^{g_1}\left(\la I_{\alpha_{j,1}}+J_{\alpha_{j,1}}\left(\mu+\frac{1}{k}\right)\right),\widehat L(\la)\right),
\end{array}
$$
where the addends of the form $\bigoplus_{j=g_\ell+1}^{g_{\ell-1}}$, with $g_{\ell-1}=g_\ell$ are empty. Note that:
\begin{itemize}
    \item $L_k(\la)\in\bun(L)$, since, following the same notation as above, $$
    {\rm KCF}(L_k)=\diag(J_{\mu+\frac{1}{k}}(\la),\hdots,J_{\mu+\frac{s}{k}}(\la),\widehat L(\la)),
    $$
    that is, $J_{\mu+\frac{i}{k}}(\la)$ is obtained from $J_{\la_i}(\la)$ just replacing the eigenvalue $\la_i$ by the eigenvalue $\mu+\frac{i}{k}$, for $i=1,\hdots,s$.
    \item The sequence $\{L_k\}_{k\in\mathbb N}$ converges to 
    $$
    \begin{array}{ccl}
\widetilde L(\la)&:=&\diag\left(\displaystyle\bigoplus_{j=1}^{g_s}\left(\bigoplus_{i=1}^s\left(\la I_{\alpha_{j,i}}+J_{\alpha_{j,i}}(\mu)\right)\right.+\right.\\&&  \left.E_{\alpha_{j,1}+\cdots+\alpha_{j,s}}(\alpha_{j,1},\alpha_{j,1}+\alpha_{j,2},\hdots,\alpha_{j,1}+\cdots+\alpha_{j,s-1})\right),\\
&&\displaystyle\bigoplus_{j=g_s+1}^{g_{s-1}}\left(\bigoplus_{i=1}^{s-1} \left(\la I_{\alpha_{j,i}}+J_{\alpha_{j,i}}(\mu)\right)\right.+\\&&  \left.E_{\alpha_{j,1}+\cdots+\alpha_{j,s-1}}(\alpha_{j,1},\alpha_{j,1}+\alpha_{j,2},\hdots,\alpha_{j,1}+\cdots+\alpha_{j,s-2})\right),\\
&&\left.,\hdots,\displaystyle\bigoplus_{j=g_2+1}^{g_1}\left(\la I_{\alpha_{j,1}}+J_{\alpha_{j,1}}(\mu)\right),\widehat L(\la)\right)\\
&=&\diag(\widetilde J_\mu(\la),\widehat L(\la)),
\end{array}
    $$
where $\mu$ is the only eigenvalue of $\widetilde J_{\mu}(\la)$, and $W(\mu,\widetilde J_\mu)=\bigcup_{i=1}^sW(\lambda_i,J_{\lambda_i})=\bigcup_{i=1}^sW(\lambda_i,L)$, by Lemma \ref{dual_lem} (note that, by construction,\break $S(\mu,\widetilde J_\mu)=\Sigma_{i=1}^s S(\la_i,L)$).
\end{itemize}
Therefore, $\widetilde L(\la) \in\overline\bun(L)$ and, moreover, it is obtained from $L(\la)$ after coalescing $\la_1,\hdots,\la_s$ to the eigenvalue $\mu$, and keeping the remaining eigenvalues of $L(\la)$ unchanged.

Using the same approach as above, applied to the block $\widehat L(\la)$, for the other values $\widehat\mu\in\overline\CC$ such that $\psi^{-1}(\widehat\mu)\cap\Lambda(\widehat L)\neq\emptyset$, we can construct a matrix pencil $L'(\la)\in\overline\bun(L)$ such that $L'(\la)=\psi_c(L)$ and $\psi^{-1}(\mu)=\{\lambda_1,\hdots,\la_s\}\cup S$, with $S\cap\Lambda(L)=\emptyset$. Since $\{\la_1,\hdots,\la_s\}$ is any subset of $\Lambda(L)$, this proves that $\psi_c(L)\in\overline\bun(L)$ for any $\psi\in\Psi$, as wanted.

Since the particular values of $\mu$ and $\widehat \mu$ are not relevant, we conclude that $\bun(\psi_c(L))\subseteq\overline\bun(L)$, for any $\psi\in\Psi$. This immediately implies $\overline\bun(\psi_c(L))\subseteq\overline\bun(L)$, by definition of closure.

Finally, if $M(\la)\in\overline\orb(\psi_c(L))$, for some $\psi\in\Psi$, then $M(\la)\in\overline\orb(\psi_c(L))\subseteq\overline\bun(\psi_c(L) )\subseteq\overline\bun(L)$, and we are done.
\end{proof}

\begin{remark}\label{bundleclosure_rem} Note that Theorem \ref{bundomprev_th} is equivalent to the identity:
   \begin{equation}\label{bundleclosureunion} \overline\bun(L)=\displaystyle\bigcup_{\psi\in\Psi}\overline{\orb}(\psi_c(L)),
    \end{equation}
    for any complex matrix pencil $L(\la)$.
\end{remark}

There is a striking difference between the statement of Theorem \ref{orbitinclusion_th} and the one of Theorem \ref{bundomprev_th}. More precisely, in Theorem \ref{orbitinclusion_th} it is implicitly assumed that $P_2(\la)\in\overline\orb(P_1)$ implies $\overline\orb(P_2)\subseteq\overline\orb(P_1)$. This is clearly true, since $P_2(\la)\in\overline\orb(P_1)$ implies that $RP_2(\la)S\in\overline\orb(P_1)$, for any pair of invertible matrices $R$ and $S$, which in turns implies $\orb(P_2)\subseteq\overline\orb(P_1)$ and this, by definition of closure, implies $\overline\orb(P_2)\subseteq\overline\orb(P_1)$. However, Theorem \ref{bundomprev_th} does not yet provide a characterization for the inclusion of bundle closures, since it is not so clear that $M(\la)\in\overline\bun(L)$ implies $\overline{\bun}(M)\subseteq\overline{\bun}(L)$. This is, indeed, true, and it is a consequence of Lemma \ref{inclusion_lemma}.

\begin{lemma}\label{inclusion_lemma}
Let $L(\la)$ and $M(\la)$ be two complex matrix pencils of the same size. If $M(\la)\in\overline\bun(L)$ then $\overline\bun(M)\subseteq\overline\bun(L)$.
\end{lemma}
\begin{proof}
We are first going to prove that, if $L(\la)$ and $M(\la)$ are as in the statement and $M'(\la)\in\bun(M)$, then $M'(\la)\in\overline\bun(L)$.

Since $M(\la)\in\overline\bun(L)$, Theorem \ref{bundomprev_th} implies that $M(\la)\in\overline{\orb}(\psi_c(L))$, for some $\psi\in\Psi$. Also, since $M'(\la)\in\bun(M)$, then $M'(\la)=R\,\varphi(K_M)\,S$, for some $R,S$ invertible and $\varphi\in\Phi$ (namely, $\varphi$ is one-to-one). In particular,  $\mu\in\Lambda(M)$ if and only if $\varphi(\mu)\in\Lambda(M')$ or, since $\varphi$ is one-to-one, $\varphi^{-1}(\mu)\in\Lambda(M)$ if and only if $\mu\in\Lambda(M')$, and, moreover:
\begin{equation}\label{varphi}
    W(\varphi^{-1}(\mu),M)=W(\mu,M').
\end{equation}

Now, we claim that 
\begin{equation}\label{inclusion}
    M'(\la)\in\overline{\orb}((\varphi\circ\psi)_c(L)).
\end{equation} 
Note that this immediately implies that $M'(\la)\in\overline\bun(L)$, by Theorem \ref{bundomprev_th} again.

In order to prove \eqref{inclusion}, we use Theorem \ref{orbitinclusion_th}. More precisely, we need to prove, for $h:=\rank(\varphi\circ\psi)_c(L)-\rank M'$:
\begin{enumerate}[{\rm (a')}]
    \item $r(M')\prec r((\varphi\circ\psi)_c(L))+(h,h,\hdots)$.
    \item $\ell(M')\prec \ell((\varphi\circ\psi)_c(L))+(h,h,\hdots)$.
    \item $W(\mu,(\varphi\circ\psi)_c(L))\prec W(\mu,M')+(h,h,\hdots)$, for all $\mu\in\overline\CC$.
\end{enumerate}
Using $M'(\la)\in\bun(M)$ and Definition \ref{coalescence_def}, we get
\begin{eqnarray}
     r(M)&=&r(M'),\label{righteq}\\\ell(M)&=&\ell(M')\label{lefteq},\\
     r(\psi_c(L))&=&r((\varphi\circ\psi)_c(L))=r(L),\label{righteq2}\\
     \ell(\psi_c(L))&=&\ell((\varphi\circ\psi)_c(L))=\ell(L).\label{lefteq2}
\end{eqnarray}
In particular, $\rank M=\rank M'$ and $\rank (\varphi\circ\psi)_c(L)=\rank L=\rank \psi_c(L)$. This implies that
\begin{equation}\label{equalh}
h=\rank\psi_c(L)-\rank M=\rank L-\rank M.    
\end{equation}
Now, to see that claims (a') and (b') are true note that, since $M(\la)\in\overline{\orb}(\psi_c(L))$, by Theorem \ref{orbitinclusion_th}, we have:
\begin{enumerate}[{\rm (a)}]
    \item $r(M)\prec r(\psi_c(L))+(h,h,\hdots)$.
    \item $\ell(M)\prec \ell(\psi_c(L))+(h,h,\hdots)$.
\end{enumerate}
Then, (a) and (b) together with \eqref{righteq}--\eqref{equalh} immediately imply (a') and (b').

To prove (c'), let $\mu\in\overline\CC$. Then, 
$$
\begin{array}{ccl}
     W(\mu,(\varphi\circ\psi)_c(L))&=& \bigcup_{\mu_i\in(\varphi\circ\psi)^{-1}(\mu)}W(\mu_i,L)=\bigcup_{\mu_i\in(\psi)^{-1}(\varphi^{-1}(\mu))}W(\mu_i,L) \\
     &=& W(\varphi^{-1}(\mu),\psi_c(L))\prec W(\varphi^{-1}(\mu),M)+(h,h,\hdots)\\&=&W(\mu,M')+(h,h,\hdots),
\end{array}
$$
where the majorization is a consequence of $M(\la)\in\overline{\orb}(\psi_c(L))$, and the last identity is \eqref{varphi}. Note that, by \eqref{equalh}, the values of $h$ in the last two expressions correspond to $\rank \psi_c(L)-\rank M$ in the first one, and to $\rank(\varphi\circ\psi)_c(L)-\rank M'$ in the second one.

We have seen so far that, if $M(\la)\in\overline\bun(L)$, then $\bun(M)\subseteq\overline\bun(L)$. By definition of closure, this implies $\overline\bun(M)\subseteq\overline\bun(L)$, as wanted.
\end{proof}

Theorem \ref{bundomprev_th}, together with Lemma \ref{inclusion_lemma}, immediately imply Theorem \ref{bundom_th}.

\begin{theorem}\label{bundom_th} {\rm (Characterization of the inclusion between bundle closures).}
Let $L(\la)$ and $M(\la)$ be two matrix pencils of the same size. Then $\overline\bun(M)\subseteq\overline\bun(L)$ if and only if $M(\la)\in\overline\orb(\psi_c(L))$, for some map $\psi:\overline\CC\rightarrow\overline\CC$.
\end{theorem}

In other words, Theorem \ref{bundom_th} says that $\overline\bun(M)\subseteq\overline\bun(L)$ if and only if ${\rm KCF}(M)$ is obtained from ${\rm KCF}(L)$ after coalescing eigenvalues and applying the dominance rules in Theorem \ref{orbitinclusion_th}. The same result is stated in \cite[Th. 2.6]{eek2} for bundles of matrices instead of matrix pencils (see Theorem \ref{bundom-matrices_th}).

\section{Bundles are open in their closure}\label{open_sec}

The fact that bundles are open in their closure will be an almost immediate consequence of the fact that $\overline{\bun}(L)$ is the union of $\bun(L)$ with a finite number of other bundles whose closure is strictly included in $\overline{\bun}(L)$. This fact is implicitly assumed in the developments carried out in \cite{eek2}. More precisely, the set of matrix pencils is claimed to be a {\em stratified manifold} (see \cite[p. 670]{eek2}), namely ``the union of non-intersecting manifolds whose closure is the finite union of itself with strata of smaller dimensions". The strata can be either orbits (for singular matrix pencils) or bundles.
Moreover, much effort has been paid, by different authors, to analyze and/or to describe the stratification of bundles (and orbits) of general and structured sets of matrices and matrix pencils (to cite just a few of the works dealing with bundles, see \cite{dk14,eek2,pervouchine} and the references therein),  as well as for matrix polynomials \cite{d17,djkvd}. In this context, a Java-based tool, called {\em Stratigraph}, has been developed \cite{ejk} for computing and displaying the closure hierarchies of the strata (bundles and orbits). However, we have not found in the literature a proof of the fact that the closure of a bundle is indeed the union of itself with other bundles of smaller dimension (for the case of orbits, it is a consequence of general results on group actions, see, for instance, the Proposition in \cite[p. 60]{humphreys}). 

The main goal of this section is to prove that bundles are open in their closure. We will prove it for bundles of matrix pencils under strict equivalence (in Theorem \ref{open_th}), as well as for bundles of matrices under similarity (in Section \ref{matrixbun_sec}) and for bundles of matrix polynomials (in Section \ref{polybun_sec}).

For bundles of matrix pencils, and as mentioned above, we are going to prove first a more general result (Theorem \ref{union_th}), namely that $\overline{\bun}(L)$ is the union of $\bun(L)$ together with a finite number of other bundles whose closure in strictly included in $\overline{\bun}(L)$. 

We start with the following well-known result, that we include here for the sake of completeness.

\begin{lemma}\label{finite-lemma}
The number of different bundles of complex matrix pencils with size $m\times n$ is finite.
\end{lemma}
\begin{proof}
Each bundle of a fixed size $m\times n$ is determined by:
\begin{itemize}
    \item The number of different eigenvalues, which is less than or equal to $\min\{m,n\}$.
    \item The sizes of the Jordan blocks associated with each eigenvalue (that is, the partial multiplicities), which are bounded by $\min\{m,n\}$.
    \item The minimal indices (left and right), whose sums are bounded by $m$ (left) and $n$ (right).
\end{itemize}
Since there is only a finite number of all these three quantities, for a fixed size $m\times n$, the number of different $m\times n$ bundles is finite.
\end{proof}

\begin{remark}
Lemma \ref{finite-lemma} is in contrast with what happens with orbits, since the number of different orbits of $m\times n$ matrix pencils is infinite (since there are infinitely many different eigenvalues).
\end{remark}

The {\em dimension} of $\orb(L)$, denoted $\dim\orb(L)$, is the dimension of $\orb(L)$ when considered as a differentiable manifold (see, for instance, \cite{de95}).

The following result provides a description of the closure of a bundle as the union of orbits (and not of orbit closures, as in \eqref{bundleclosureunion}), and is a first step to prove Theorem \ref{union_th}, where the closure of a bundle is described as a union of bundles.

\begin{theorem}\label{bundorbitunion_th}
Let $L(\la)$ be an $m\times n$ complex matrix pencil. Then $\overline\bun(L)=\bun(L)\cup\widetilde \bun$, where $\widetilde \bun$ is a union of orbits. Each of these orbits, say $\orb(M)$, satisfies one of the following conditions:
\begin{itemize}
    \item $\dim\orb(M)<\dim\orb(L)$, or
    \item $\orb(M)=\orb(\psi_c(L))$, for some $\psi\in\Psi$ which is not one-to-one on $\Lambda(L)$.
\end{itemize} 
\end{theorem}
\begin{proof}
Let us first prove the following identity:
\begin{equation}\label{dimorb}
    \dim\orb(\psi_c(L))=\dim\orb(L),\quad \mbox{for all $\psi\in\Psi$}.
\end{equation}
If $\psi\in\Phi$, then the identity in \eqref{dimorb} is an immediate consequence of the codimension count in \cite{de95}. More precisely, the codimension (and, as a consequence, the dimension) of the orbit of a pencil $L(\la)$ depends on the sizes of the regular and singular blocks in KCF($L$), in such a way that for two pencils in the same bundle the (co)dimensions of their orbits are the same (see the quantities 1--5 in the codimension count of \cite[p. 65]{de95}). Also all quantities except the one numbered by 1 (and so-called ``codimension of the Jordan structure") in \cite[p. 65]{de95} are the same for $L(\la)$ and $\psi_c(L)$, even when $\psi$ is not one-to-one, because these two pencils have the same singular blocks, and the size of their regular part (the ``Jordan structure") is the same as well. A closer look to quantity 1 allows us to conclude that it is also the same for $L(\la)$ and $\psi_c(L)$ when $\psi$ is not one-to-one. To see this, let us first recall that this quantity is equal to
$$
c_{Jor}(L)=\sum_{\mu\in\overline{\CC}}\left(S_1(\mu,L)+3S_2(\mu,L)+5S_3(\mu,L)+\cdots\right),
$$
where $S(\mu,L)=(S_1(\mu,L),S_2(\mu,L),S_3(\mu,L),\hdots)$ is the Segre characteristic of $\mu$ in $L(\la)$, namely, the list of sizes of the Jordan blocks in KCF($L$) associated with $\mu$, ordered in non-increasing order. If $\psi^{-1}(\mu)=\{\mu_1,\hdots,\mu_k\}\cup S$, with $S\cap\Lambda(L)=\emptyset$, then, by Definition \ref{coalescence_def}, $W(\mu,\psi_c(L))=\bigcup_{i=1}^k W(\mu_i,L)$, which is equivalent to $S(\mu,\psi_c(L))=\sum_{i=1}^k S(\mu_i,L)$, by Lemma \ref{dual_lem}. Then
$$
\begin{array}{ccl}
c_{Jor}(\psi_c(L))&=&\sum_{\mu\in\overline{\CC}}\left(S_1(\mu,\psi_c(L))+3S_2(\mu,\psi_c(L))+5S_3(\mu,\psi_c(L))+\cdots\right)\\&=&\sum_{\mu\in\overline{\CC}}\left(S_1(\mu,L)+3S_2(\mu,L)+5S_3(\mu,L)+\cdots\right)=c_{Jor}(L).
\end{array}
$$

Now, let us decompose the identity in \eqref{bundleclosureunion} as follows:
\begin{equation}\label{twounion}
\overline\bun(L)=\left(\bigcup_{\varphi\in\Phi}\overline\orb(\varphi(L))\right)\cup\left(\bigcup_{\psi\in\Psi_{\Lambda(L)}}\overline\orb(\psi_c(L))\right),
\end{equation}
where $\Psi_{\Lambda(L)}$ is the set of mappings from $\overline{\CC}$ to itself which are not one-to-one on $\Lambda(L)$. In order to obtain \eqref{twounion} from \eqref{bundleclosureunion}, just note that $\varphi_c(L)=\varphi(L)$ when $\varphi\in\Phi$, so the decomposition in \eqref{twounion} comes from splitting the union in \eqref{bundleclosureunion} into two pieces: one corresponding to the maps from $\overline\CC$ to $\overline{\CC}$ which are one-to-one on $\Lambda(L)$, and the other one corresponding to the remaining maps. Now, since the closure of an orbit is the union of the orbit itself and other orbits with smaller dimension (see the Proposition in \cite[p. 60]{humphreys}), we can write
$$
\overline\bun(L)=\bun(L)\cup \widetilde \bun_1\cup\widetilde \bun_2\cup\widetilde \bun_3,
$$
where 
$$
\bun(L)=\bigcup_{\varphi\in\Phi}\orb(\varphi(L)),\quad\widetilde \bun_1=\bigcup_{\psi\in\Psi_{\Lambda(L)}}\orb(\psi_c(L)),
$$
the set $\widetilde \bun_2$ is a union of orbits with dimension smaller than $\orb(\varphi(L))$, for any $\varphi\in\Phi$ (coming from the left union in the right-hand side of \eqref{twounion}), and $\widetilde \bun_3$ is a union of orbits with dimension smaller than $\orb(\psi_c(L))$, for any $\psi\in\Psi_{\Lambda(L)}$ (coming from the right union in the right-hand side of \eqref{twounion}).

By \eqref{dimorb}, the orbits in $\widetilde B_2$ have dimension strictly smaller than the dimension of $\orb(L)$. Similarly, $\widetilde\bun_1$ is a union of orbits of the form $\orb(\psi_c(L))$, they all having the same dimension as $\orb(L)$, and $\widetilde \bun_3$ is a union of orbits with dimension strictly smaller than the dimension of $\orb(L)$. This concludes the proof.
\end{proof}

The following result provides a description of the closure of the bundle of a pencil as a finite union of bundles. This is the counterpart of the corresponding result for orbits, namely that the closure of an orbit is a union of orbits, which is already known (and has been used in the proof of Theorem \ref{bundorbitunion_th}). Nevertheless, we emphasize the important difference that the closure of the orbit of a pencil may be the union of infinitely many orbits.

\begin{theorem}\label{union_th}
Let $L(\la)$ be an $m\times n$ pencil. Then, there is a finite number of different $m\times n$ pencils, $L_1(\la),\hdots,L_\ell(\la)$ (with, say, $L_1=L$), satisfying $\bun(L_i)\neq\bun(L_j)$, for $i\neq j$, and such that 
\begin{equation}\label{bundleclosure}
    \overline\bun(L)=\bigcup_{i=1}^\ell \bun(L_i).
\end{equation}
Moreover, 
$\overline\bun(L_i)$ is strictly included in $\overline\bun(L)$, for $i\neq1$.
\end{theorem}
\begin{proof} 
Since, by Lemma \ref{finite-lemma}, the number of different bundles of $m\times n$ matrix pencils is finite, let $\bun(L_1),\hdots,\bun(L_d)$ be these bundles, for some $m\times n$ matrix pencils $L_1(\la),\hdots,L_d(\la)$. Without loss of generality (after a reordering, if necessary), we may assume that $L_1(\la),\hdots,L_\ell(\la)\in\overline\bun(L)$, and $L_{\ell+1}(\la),\hdots,L_d(\la)\not\in\overline\bun(L)$, for some $0\leq\ell\leq d$. By Lemma \ref{inclusion_lemma}, $\bigcup_{i=1}^\ell\bun(L_i)\subseteq\overline\bun(L)$. The reverse inclusion is immediate, since $\widetilde L(\la)\in\overline\bun(L)$ implies that $\bun(\widetilde L)\subseteq\overline\bun(L)$, again by Lemma \ref{inclusion_lemma}, so $\bun(\widetilde L)=\bun(L_i)$, for some $1\leq i\leq \ell$, so $\widetilde L(\la)\in\bigcup_{i=1}^\ell\bun(L_i)$.


 Finally, let us prove that $\overline\bun(L_i)$ is strictly included in $\overline\bun(L)$, for $i\neq1$. First, note that, by Theorem \ref{bundorbitunion_th}, either
\begin{itemize}
    \item[{\rm (a)}] $\dim\orb(L_i)<\dim\orb(L)$, or
    \item[{\rm(b)}] $\orb(L_i)=\orb(\psi_c(L))$, for some $\psi\in\Psi$ which is not one-to-one on $\Lambda(L)$.
\end{itemize}
In case (a) we clearly have
\begin{equation}\label{dimineq}
\dim\orb(L_i)\leq\dim\orb(L),
\end{equation}
and \eqref{dimorb} implies that \eqref{dimineq} also holds in case (b). 

Assume, by contradiction, that $\overline{\bun}(L_i)=\overline{\bun}(L)$, for some $i\neq1$. This implies, in particular, that $L(\la)\in\overline\bun(L_i)$. Thus, Theorem \ref{bundorbitunion_th} applied to $\overline\bun(L_i)$ implies that either
\begin{itemize}
    \item[{\rm (c)}] $\dim\orb(L)<\dim\orb(L_i)$, or
    \item[{\rm (d)}] $\orb(L)=\orb(\widetilde\psi_c(L_i))$, for some $\widetilde\psi\in\Psi$ which is not one-to-one on $\Lambda(L_i)$.
\end{itemize}
Condition (c) does not hold because of \eqref{dimineq}, so (d) must hold. But (d) combined with \eqref{dimorb} implies 
$\dim\orb(L)=\dim\orb(\widetilde\psi_c(L_i))=\dim\orb(L_i)$, which implies that condition (a) does not hold. So, (b) and (d) hold simultaneously. But (b) implies that the number of distinct eigenvalues of $L_i$ is strictly less than the number of distinct eigenvalues of $L$, while (d) implies the opposite, which is a contradiction. Thus, we conclude that $\overline\bun(L_i)$ is strictly included in $\overline\bun(L)$.
\end{proof}

Now, we are in the position to prove the main result in this section.

\begin{theorem}\label{open_th}
Let $L(\la)$ be an $m\times n$ matrix pencil. Then $\bun(L)$ is an open set in its closure.
\end{theorem}
\begin{proof}
Let $M(\la)\in\bun(L)$. We want to prove that there is an open neighborhood of $M(\la)$, say $U$, such that $U\cap \overline\bun(L)\subseteq\bun(L)$. Let us proceed by contradiction. 
Then, there is a sequence, $\{M_k(\la)\}_{k\in\mathbb N}$, of matrix pencils that converges to $M(\la)$, with $M_k(\la)\in\overline\bun(L)$ but $M_k(\la)\not\in\bun(L)$.
By Theorem \ref{union_th}, $\overline\bun(L)=\bigcup_{i=1}^\ell\bun(L_i)$, with $L_1(\la)=L(\la)$. Since this is a finite union and $M_k(\la)\not\in\bun(L)$, there is an infinite number of terms in the sequence $\{M_k(\la)\}_{k\in\mathbb N}$
belonging to some $\bun(L_i)$, with $i\neq1$. Without loss of generality, let us assume that $i=2$ and, considering a subsequence if necessary, assume that $M_k(\la)\in\bun(L_2)$, for all $k\in\mathbb N$. But $M_k(\la)\rightarrow M(\la)$ implies that $M(\la)\in\overline\bun(L_2)$. Note that the hypothesis $M(\la)\in\bun(L)$ implies $\bun(M)=\bun(L)$, so Lemma \ref{inclusion_lemma} in turn implies that $\overline\bun(L)\subseteq \overline\bun(L_2)$, which in turn implies $\overline\bun(L)= \overline\bun(L_2)$, and this is a contradiction with the last claim in Theorem \ref{union_th}. 
\end{proof}

\begin{remark}\label{pervouchine_remark}
We want to highlight that Theorem 7.5 in \cite{pervouchine} is false. This theorem states the following, with the notation of this paper:
\begin{equation}\label{bundunion}
    \overline\bun(L)=\bigcup_{\psi\in\Psi}\overline\orb(\psi(L)).
\end{equation}
Comparing \eqref{bundunion} with \eqref{bundleclosureunion}, we see that the difference is that in \eqref{bundunion} the function $\psi$ is used instead of the right one, namely $\psi_c(L)$. In other words, \eqref{bundunion} is missing all pencils that are obtained by coalescing eigenvalues of $L(\la)$.

A counterexample of \eqref{bundunion} is the following. Let
$$
L(\la):=\left[\begin{array}{c|cc}
     \la-3&0&0  \\\hline
     0&\la-2&1\\0&0&\la-2 
\end{array}\right]=\diag(\la +J_1(3),\la I_2+J_2(2)),
$$
and
$$
L_k(\la):=\begin{bmatrix}\la-(2+1/k)&1&0\\0&\la-2&1\\0&0&\la-2\end{bmatrix}.
$$
Note that
$$
\widehat L(\la):=\lim_{k\rightarrow\infty}L_k(\la)=\begin{bmatrix}\la-2&1&0\\0&\la-2&1\\0&0&\la-2\end{bmatrix}=\la I_3+J_3(2).
$$
Then 
$$
L_k(\la)\in\orb\left(\left[\begin{array}{c|cc}\la-(2+1/k)&0&0\\\hline0&\la-2&1\\0&0&\la-2\end{array}\right]\right)\subseteq \bun(L),\quad\mbox{for all $k\in\mathbb N$,}
$$
which implies that 
\begin{equation}\label{in}
\widehat L(\la)\in\overline\bun(L).
\end{equation}

We are going to prove that 
\begin{equation}\label{notin}
\widehat L(\la)\not\in \bigcup_{\psi\in\Psi}\overline\orb(\psi(L)).
\end{equation}
Note that equations \eqref{in} and \eqref{notin} contradict \eqref{bundunion}.

In order to see \eqref{notin}, note first that
$$
\orb(\psi(L))=\orb\left(\left[\begin{array}{c|cc}
     \la-\psi(3)&0&0  \\\hline 0&\la-\psi(2)&1\\0&0&\la-\psi(2)
\end{array}\right]\right)=\orb\left(\left[\begin{array}{c|cc}
    \la- a&0&0  \\\hline 0&\la-b&1\\0&0&\la-b
\end{array}\right]\right),
$$
with $a,b\in\overline\CC$ arbitrary (if $a$ or $b$ is $\infty$ then the pencil in the right-hand side should be replaced accordingly by either $\diag(\la-a,\la N_2+I_2)$, $\diag(1,\la I_2+J_2(b)),$ or $\diag(1,\la N_2+I_2)$).

If $a=b$, then 
$$
\widehat L(\la)\not\in \overline\orb(\left(\left[\begin{array}{c|cc}
     \la-a&0&0  \\\hline 0&\la-a&1\\0&0&\la-a
\end{array}\right]\right)
$$
(with, again, performing the appropriate changes if either $a=b=\infty$). To see this, note that, in order for $\widehat L(\la)$ to belong to the closure of the orbit in the right-hand side of the previous equation, it should be $a=2$. However, the closure relationships of matrix orbits in Theorem \ref{orbitinclusion_th} guarantee that $\widehat L$ does not belong to this closure (the majorization $(2,1,0)\prec(1,1,1)$ does not hold).

If $a\neq b$, then 
$$
\widehat L(\la)\not\in \overline\orb(\left(\left[\begin{array}{c|cc}
     \la-a&0&0  \\\hline 0&\la-b&1\\0&0&\la-b
\end{array}\right]\right)
$$
because the eigenvalues of regular pencils are continuous functions of the entries of their coefficients, thus $\widehat L(\la)$, that has only one eigenvalue equal to $2$ (with multiplicity $3$) cannot be the limit of a sequence of pencils with two different eigenvalues equal to $a$ and $b$ (fixed).
\end{remark}

\subsection{Bundles of matrices under similarity}\label{matrixbun_sec}

We can also prove that bundles of matrices under similarity are open in their closure using similar arguments to the ones for bundles of matrix pencils above in this section, combined with results from \cite{eek2}. More precisely, for a matrix $A\in\CC^{n\times n}$, we define
$$
\begin{array}{ccll}
     \orb(A)&:=&\{PAP^{-1}\,:\ P\in\mbox{\rm GL}_n(\CC)\} &\mbox{(orbit of $A$),}\\
     \bun(A)&:=&\displaystyle\bigcup_{\varphi\in\Phi}\orb(\varphi(A)) &\mbox{(bundle of $A$),}
\end{array}
$$
where now $\Phi$ is the set of one-to-one maps from $\CC$ to itself (since the infinite eigenvalue does not apply for matrices) and $\varphi(A)$ is any matrix similar to the one obtained from the JCF of $A$ after replacing the Jordan blocks with eigenvalue $\mu \in \mathbb{C}$ by Jordan blocks of the same size with eigenvalue $\varphi(\mu)$ for any eigenvalue $\mu$ of $A$. Similarly, if $\Psi$ denotes the set of maps from $\CC$ to itself (not necessarily one-to-one), then we can introduce the following notion of coalescence for matrices:

\begin{definition}\label{coalescence-matrices_def} {\rm (Coalescence of eigenvalues of matrices).}
Let $A\in\CC^{n\times n}$ be a matrix  with different eigenvalues $\mu_1,\hdots,\mu_s$, and let $\psi\in\Psi$. Then, $\psi_c(A)$ is any $n\times n$ matrix satisfying the following property:
\begin{itemize}
    \item $W(\mu,\psi_c(A))=\bigcup_{\mu_i\in\psi^{-1}(\mu)}W(\mu_i,A)$, for all $\mu\in\CC$.
\end{itemize}
We say that the eigenvalues $\mu_{i_1},\hdots,\mu_{i_d}$ of $A$ have {\em coalesced} to the eigenvalue $\mu$ in $\psi_c(A)$ if $\psi^{-1}(\mu) = \{\mu_{i_1},\hdots,\mu_{i_d}\}\cup S$, with $S\cap \Lambda (A) = \emptyset$.
\end{definition}
Observe that all matrices $\psi_c(A)$ in Definition \ref{coalescence-matrices_def} are similar to each other, since they all have the same Jordan canonical form.

With these definitions, the characterization of the bundle closure inclusion for matrices under similarity, provided in \cite[Th. 2.6]{eek2}, can be stated as follows:

\begin{theorem}\label{bundom-matrices_th}{\rm \cite[Th. 2.6]{eek2}} {\rm (Characterization of the inclusion between bundle closures of matrices).}
Let $A_1,A_2\in\CC^{n\times n}$. Then $\overline\bun(A_2)\subseteq\overline\bun(A_1)$ if and only if $A_2\in\overline\orb(\psi_c(A_1))$, for some map $\psi:\CC\rightarrow\CC$.
\end{theorem}

The statement of Theorem \ref{bundom-matrices_th} is implicitly assuming that $A_2\in\overline\bun(A_1)$ implies $\overline{\bun}(A_2)\subseteq\overline\bun(A_1)$ (the counterpart of Lemma \ref{inclusion_lemma} for matrices under similarity), but this can be proved using similar arguments to the ones in the proof of Lemma \ref{inclusion_lemma}, replacing matrix pencils by matrices, the KCF by the JCF, and Theorem \ref{orbitinclusion_th} by the corresponding one for orbits of matrices under similarity, which only includes condition (iii) with $h=0$. Also, using similar arguments to the ones in the proof of Theorem \ref{union_th}, including the fact that the number of different bundles of $n\times n$ matrices is finite (the counterpart of Lemma \ref{finite-lemma} for bundles of matrices under similarity), we can prove the counterpart of Theorem \ref{union_th} for bundles of matrices
\begin{theorem}\label{union-matrices_th}
Let $A\in\CC^{n\times n}$. Then, there is a finite number of different matrices, $A_1,\hdots,A_\ell\in\CC^{n\times n}$ (with, say, $A_1=A$), such that 
$$
    \overline\bun(A)=\bigcup_{i=1}^\ell \bun(A_i).
$$
Moreover, $\overline\bun(A_i)$ is strictly included in $\overline\bun(A)$, for $i\neq1$.
\end{theorem}

Finally, from Theorem \ref{union-matrices_th} and using similar arguments to the ones in the proof of Theorem \ref{open_th}, we can prove that bundles of matrices under similarity are open in their closure:

\begin{theorem}\label{open-matrices_th}
Let $A\in\CC^{n\times n}$. Then $\bun(A)$ is an open set in its closure.
\end{theorem}

\subsection{Bundles of matrix polynomials of higher degree}\label{polybun_sec}

From the previous developments, we can also conclude that bundles of matrix polynomials are open in their closure.

A matrix polynomial of degree $d$ in the variable $\la$, say $P(\la)$, is of the form 
$
P(\la)=\sum_{i=0}^d\la^{i}A_i,
$
with $A_i\in\CC^{m\times n}$, for $0\leq i\leq d$, and $A_d\neq0$. If we allow $A_d$ to be zero, then we say that $P(\la)$ has {\em grade} $d$. Matrix pencils are particular cases of matrix polynomials (namely, when $d=1$). 

We denote by ${\rm POL}_d^{m\times n}$ the set of complex matrix polynomials of size $m\times n$ and grade $d$.

The {\em orbit} of a matrix polynomial $P(\la)$ is the subset of matrix polynomials in ${\rm POL}_d^{m\times n}$ having exactly the same complete spectral information as $P(\la)$ (see \cite[p. 217]{dd17}), i. e., the same eigenvalues with the same partial multiplicities, and the same left and right minimal indices. The definitions of these concepts can be found, for instance, in \cite{dd17} and the references therein. Similarly, the {\em bundle} of $P(\la)$, denoted by $\bun(P)$, is the subset of matrix polynomials in ${\rm POL}_d^{m\times n}$ having the same spectral information as $P(\la)$, ``except that the values of the distinct eigenvalues are unspecified as long as they remain distinct" (this sentence is taken from \cite{ddd-polys}, for the case of symmetric matrix polynomials).  We want to emphasize that, for matrix polynomials of higher grade, the action of strict equivalence is not appropriate for defining orbits in this context, since two matrix polynomials of grade larger than $1$ can have the same spectral information without being strictly equivalent.

If ${\cal C}_P(\la)$ denotes the first Frobenius companion linearization of $P(\la)$ (see, for instance, \cite{dd17}), then $Q(\la)\in{\rm POL}_{d}^{m\times n}$ belongs to the bundle of $P(\la)$ if and only if ${\cal C}_Q(\la)\in\bun({\cal C}_P)$. This is because $Q(\la)$ has the same complete spectral information as $P(\la)$ if and only if ${\cal C}_P(\la)$ has the same KCF as ${\cal C}_Q(\la)$. Moreover, there is a homeomorphism,
\begin{equation}\label{homeo}
\begin{array}{cccc}
     f:&{\rm POL}_d^{m\times n}&\rightarrow&{\rm GSYL}_d^{m\times n}  \\
     &P(\la)&\mapsto&{\cal C}_P(\la), 
\end{array}
\end{equation}
where ${\rm GSYL}_d^{m\times n}$ denotes the set of complex matrix pencils of the form ${\cal C}_P(\la)$, with $P(\la)\in{\rm POL}_d^{m\times n}$ (the so-called {\em generalized Sylvester space} in \cite{dd17}). The fact that \eqref{homeo} is a homeomorphism is immediate because it is an isometry (for instance, with the following distance for matrix pencils and matrix polynomials, defined in terms of the Frobenius norms of the coefficients, $\rho(\sum_{i=0}^d\la^{i}A_i,\sum_{i=0}^d\la^{i}A_i'):=\left(\sum_{i=0}^d\|A_i-A_i'\|_F\right)^{1/2}$, see \cite{dd17}). 

Now, we define (see Eq. (5.4) in \cite{ddd-polys} for symmetric matrix polynomials and the congruence relation)
$$
\bun^{\rm syl}({\cal C}_P):=\bun({\cal C}_P)\cap{\rm GSYL}_{d}^{m\times n}.
$$
Then, $\bun^{\rm syl}({\cal C}_P)$ is also open in its closure. This is a consequence of $\bun({\cal C}_P)$ being open in its closure (Theorem \ref{open_th}). More precisely, one way to see this is the following: Let $M\in\bun^{\rm syl}({\cal C}_P)=\bun({\cal C}_P)\cap{\rm GSYL}_{d}^{m\times n}$. Then, $M(\la)\in\bun({\cal C}_P)$, so there is a neighborhood of $M(\la)$, say $U_M$, such that $U_M\cap\overline\bun({\cal C}_P)\subseteq\bun({\cal C}_P)$. Now
$$
\begin{array}{ccl}
U_M\cap\overline\bun^{\rm syl}({\cal C}_P)&=&U_M\cap \overline{\bun({\cal C}_P)\cap{\rm GSYL}_{d}^{m\times n}}\subseteq U_M \cap \overline\bun({\cal C}_P)\cap{\rm GSYL}_d^{m\times n}\\
&\subseteq&\bun({\cal C}_P)\cap{\rm GSYL}_d^{m\times n}=\bun^{\rm syl}({\cal C}_P),
\end{array}
$$
 so there is also a neighborhood of $M(\la)$ whose intersection with $\overline\bun^{\rm syl}({\cal C}_P)$ is contained in $\bun^{\rm syl}({\cal C}_P)$, which means that $\bun^{\rm syl}({\cal C}_P)$ is open in its closure.
 
Since \eqref{homeo} is a homeomorphism, we have
$$
f^{-1}(\bun^{\rm syl}({\cal C}_P))=\bun(P),\quad\mbox{and}
\quad f^{-1}\left(\overline\bun^{\rm syl}({\cal C}_P)\right)=\overline\bun(P).
$$
(see \cite[p. 1047]{ddd-polys} for symmetric matrix polynomials). As a consequence, the property of $\bun^{\rm syl}({\cal C}_P)$ being open in its closure can be translated to $\bun(P)$. More precisely, let $M(\la)\in\bun(P)$. Then $f(M)={\cal C}_M(\la)\in\bun^{\rm syl}({\cal C}_P)$, and, as we have just seen, there is an open neighborhood of ${\cal C}_M$, say $U_M$, such that $U_M\cap\overline\bun^{\rm syl}({\cal C}_P)\subseteq\bun^{\rm syl}({\cal C}_P)$. Since $f$ is a homeomorphism, $f^{-1}(U_M)$ is an open neighborhood of $M$, say $\widetilde U_M$, so that
$$
f^{-1}\left(U_M\cap\overline\bun^{\rm syl}({\cal C}_P)\right)=\widetilde U_M\cap f^{-1}(\overline\bun^{\rm syl}({\cal C}_P))=\widetilde U_M \cap \overline\bun(P),
$$
and
$$
f^{-1}(U_M\cap\overline\bun^{\rm syl}({\cal C}_P))\subseteq f^{-1}( \bun^{\rm syl}({\cal C}_P))=\bun(P)
$$
together imply that $\widetilde U_M\cap \overline\bun(P)\subseteq\bun(P)$, for every $M(\la)\in{\cal B}(P)$, so $\bun(P)$ is indeed open in its closure.

We have then proved the following result.

\begin{theorem}\label{open-poly_th}
Let $P(\la)\in{\rm POL}_d^{m\times n}$. Then, the bundle of $P(\la)$ (namely, the set of matrix polynomials in ${\rm POL}_d^{m\times n}$ with the same complete spectral information as $P(\la)$, up to the specific values of the distinct eigenvalues) is an open set in its closure.
\end{theorem}

\section{Conclusions and open questions}\label{conclusion_sec}

This paper is mainly devoted to bundles of matrix pencils under strict equivalence. We have provided necessary and sufficient conditions for the closure of a given bundle to be included in the closure of another one (Theorem \ref{bundom_th}). A proof that bundles are open in their closure is also given (see Theorem \ref{open_th}). The same has been done for bundles of matrices under similarity (in Theorem \ref{open-matrices_th}) and bundles of matrix polynomials of higher degree (in Theorem \ref{open-poly_th}). We have also revisited some notions already present in the literature (like the one of ``coalescence"), as well as some other previously known results. Some additional technical results, that can be useful in the future, have been proved.

Some further lines of research that naturally arise as a continuation of the present work are the following:

\begin{itemize}
    \item We believe that, if $L(\la)$ and $L_i(\la)$ are as in Theorem \ref{union_th}, then $\dim\bun(L_i)<\dim\bun(L)$ for $i\ne 1$, with
    \begin{equation*}\label{dimbundle}\dim\bun(L):=\dim\orb(L)+\#\{\mbox{different eigenvalues of $L(\la)$}\}\end{equation*}
    (see, for instance, \cite{djkvd}). This would formally prove that $\overline{\bun}(L)$ is the union of $\bun(L)$ together with a finite number of bundles with smaller dimension. A natural approach to prove this fact is by using the breakdown of the codimension count of pencil orbits provided in \cite{de95}.
    
    \item Given two matrix polynomials, $P_1(\la)$ and $P_2(\la)$, provide necessary and sufficient conditions for $\overline\bun(P_2)\subseteq\overline\bun(P_1)$.
    
    \item Given two structured matrix pencils (or, more in general, matrix polynomials of higher degree) $L_1(\la)$ and $L_2(\la)$, for any of the structures mentioned in Section \ref{intro_sec} (namely, {\em alternating}, {\em (skew-) Hermitian}, {\em (anti-) palindromic}, or {\em (skew-) symmetric}), provide necessary and sufficient conditions for $\overline\bun(L_2)\subseteq\overline\bun(L_1)$. The bundles in this case should be defined for congruence or $*$-congruence, but the definition in this case will probably deserve a more detailed care, because of the restrictions and symmetries in the spectral information that are imposed by the structure (see, for instance, \cite{4m-good}).
    
    \item To prove whether or not bundles of structured matrix pencils (or matrix polynomials of higher degree) are open in their closure. 
    This is know to be true, for instance, for the bundles of generic symmetric matrix pencils and matrix polynomials with bounded rank, described in \cite{ddd-polys}.
\end{itemize}

\noindent{\bf Acknowledgments.} The authors are very much indebted to Inmaculada de Hoyos and Juan Miguel Gracia for providing the source \cite{tesis-inma}, that has been key to prove Theorem \ref{coalescence_th} and for helpful discussions on this result, as well as for pointing reference \cite{macdonald} out to us.

\end{document}